\documentclass[letterpaper,10pt]{amsart}
\usepackage{amsmath,amssymb,amsxtra,amsthm, amstext, amscd,amsfonts,fancyhdr, hyperref,enumerate}
\usepackage[OT2,T1]{fontenc}
\DeclareSymbolFont{cyrletters}{OT2}{wncyr}{m}{n}
\DeclareMathSymbol{\Sha}{\mathalpha}{cyrletters}{"58}

\newcommand{\bP}{{\mathbb{P}}}
\newcommand{\bQ}{{\mathbb{Q}}}
\newcommand{\bR}{{\mathbb{R}}}

\newcommand{\bZ}{{\mathbb{Z}}}

\newcommand{\Bx}{{\mathbf{x}}}

\newcommand{\A}{{\mathcal{A}}}

\newcommand{\C}{{\mathcal{C}}}

\newcommand{\E}{{\mathcal{E}}}
\newcommand{\F}{{\mathcal{F}}}
\newcommand{\G}{{\mathcal{G}}}

\newcommand{\M}{{\mathcal{M}}}
\newcommand{\N}{{\mathcal{N}}}

\renewcommand{\P}{{\mathcal{P}}}
\newcommand{\Q}{{\mathcal{Q}}}
\newcommand{\R}{{\mathcal{R}}}
\renewcommand{\S}{{\mathcal{S}}}

\newcommand{\X}{{\mathcal{X}}}

\newcommand{\Z}{{\mathcal{Z}}}

\newcommand{\fJ}{\mathfrak{J}}

\newcommand{\ind}{\operatorname{ind}}

\newcommand{\Sym}{\operatorname{Sym}}

\newcommand{\GL}{\operatorname{GL}}

\newcommand{\SO}{\operatorname{SO}}

\newcommand{\ep}{\varepsilon}

\newcommand{\rad}{\operatorname{rad}}
\newcommand{\KT}{\operatorname{III}}
\newcommand{\KTS}{\operatorname{III}^\ast}  
\newcommand{\KTI}{\operatorname{I}_0^\ast}

\newcommand{\ol}{\overline}

\newcommand{\upchi}{{\raise.35ex\hbox{$\chi$}}}

\newcommand{\SL}{\operatorname{SL}}
\newcommand{\Vol}{\operatorname{Vol}}

\makeatletter
\@namedef{subjclassname@2010}{%
	\textup{2010} Mathematics Subject Classification}
\makeatother

\newtheorem{theorem}{Theorem}[section]

\newtheorem{proposition}[theorem]{Proposition}
\newtheorem{lemma}[theorem]{Lemma}

\theoremstyle{definition}

\numberwithin{equation}{section}

\begin{document}
	
	\title[Elliptic curves with a rational 2-torsion point ordered by conductor ]{Elliptic curves with a rational 2-torsion point ordered by conductor and the boundedness of average rank}

	\indent
	
	\author{Stanley Yao Xiao}
	\address{Department of Mathematics and Statistics \\
		University of Northern British Columbia \\
		3333 University Way \\
		Prince George, British Columbia, Canada \\  V2N 4Z9}
	\email{sxiao@unbc.ca}
	\indent
	
	
	\begin{abstract} In this paper we refine recent work due to A.~Shankar, A.~N.~ Shankar, and X.~Wang on counting elliptic curves by conductor to the case of elliptic curves with a rational 2-torsion point. This family is a \emph{small} family, as opposed to the large families considered by the aforementioned authors. We prove the analogous counting theorem for elliptic curves with so-called square-free index as well as for curves with suitably bounded Szpiro ratios. We note that our assumptions on the size of the Szpiro ratios is less stringent than would be expected by the naive generalization of their approach; this requires an application of a theorem of Browning and Heath-Brown. Our proof also relies on linear programming techniques. 
	\end{abstract}
	
	\maketitle

	\section{Introduction}
	\label{Intro}
	
	In this paper we consider the problem of counting elliptic curves and estimating their average rank in certain thin families ordered by their conductor. The families we consider will be sub-families of the family $\E_2$ of elliptic curves with a rational 2-torsion point (or equivalently, those that admit a degree 2 isogeny over $\bQ$) and having good reduction at $2$ and $3$. The latter is a technical assumption that simplifies the exposition. For this, we use the work of J. Mulholland on characterizing the reduction types at $2$ and $3$ of elliptic curves with rational 2-torsion \cite{Mul}.\\ 
	
	Aside from the work \cite{SSW}, there has been a recent explosion of interest in enumerating elliptic curves ordered by their conductors, inspired by work of He, Lee, Oliver, Pozdnyakov \cite{HLOP} and Zubrilina \cite{Zub}. \\
	
	We may assume without loss of generality, that our 2-torsion point is \emph{marked} and so it suffices to consider the parametrization 
	\begin{equation} \label{main fam} E_{a,b} : y^2 = x(x^2 + ax + b) : a,b \in \bZ.
	\end{equation}
	The discriminant of the curves in this family is given by 
	\begin{equation} \label{disc formula 1} \Delta(E_{a,b}) = 16 b^2 (a^2 - 4b).
	\end{equation}	
	The \emph{conductor} of the curve is the quantity $C(E_{a,b})$ given in terms of $p$-adic valuation by 
	\begin{equation}
	 v_p(C(E_{a,b})) =  \begin{cases} 0 & \text{if } p \nmid \Delta(E_{a,b}) \\ 
	 1 & \text{if } E_{a,b} \text{ has multiplicative bad reduction at } p \\
	 2 & \text{if } E_{a,b} \text{ has additive bad reduction at } p. \end{cases}	\end{equation}
	
	Our goal in this paper is to count certain subfamilies of elliptic curves of our families in $\E_2$ ordered by conductor, as well as estimating the average rank with respect to such an ordering. This is motivated by recent work of A.~N.~Shankar, A.~Shankar, and X.~Wang \cite{SSW} on counting elliptic curves in large families having bounded conductor. They also showed that the average size of the 2-Selmer group in the families they consider is at most $3$.   \\
	
	Before we state our theorems, let us make a comparison with the approach and results in \cite{SSW}. In their treatment they deal with families which are conjectured to have positive density. In order to make progress they make a key assumption which we do not need: they assumed that the $j$-invariant $j(E)$ of their elliptic curves $E$ is bounded by $O(\log |\Delta(E)|)$. In particular, this allows them to remove the well-known archimedean difficulties of counting elliptic curves by discriminant. Another idea critical to their argument is to embed their families into the space of binary quartic forms. As a corollary they are able to count the corresponding 2-Selmer elements with a bit of extra work. Since we do not require this embedding, and indeed the strength of our results depend on our direct treatment of the elliptic curves under consideration, we do not obtain an analogous count of 2-Selmer elements. In fact we expect such a count to be useless for our purposes: most of the 2-Selmer elements we encounter for the family $\E_2$ are expected to correspond to elements of the Shafarevich-Tate group $\Sha$; see work of Klagsbrun and Lemke-Oliver \cite{K-LO} as well as recent work of Bhargava and Ho \cite{B-Ho}. Therefore in order to obtain an estimate for the average rank we instead look at the \emph{3-Selmer group}, employing the parametrization of 3-Selmer elements of curves in our family due to Bhargava and Ho \cite{B-Ho}. \\
		
	There is a significant difference between the results of \cite{SSW} and our results. In particular, Shankar, Shankar, and Wang obtain theorems where counting by conductor produces the same order of magnitude as counting by discriminant: in other words, on average the conductor is only marginally smaller than the discriminant in the cases they consider. For us, however, we obtain substantially larger counts when ordering our curves by conductor rather than discriminant: this phenomenon is again caused by the special shapes of our discriminants. Our previous paper \cite{TX} with C.~Tsang provides another example of this phenomenon. \\ 
	
	To motivate our results, let us discuss the analogous results of \cite{SSW} in more detail. First we recall, as discussed in \cite{SSW}, that their strategy involves first counting elliptic curves by their \emph{discriminant} first. This is where they needed to assume that the $j$-invariants of the curves $E$ under consideration are bounded by $\log |\Delta(E)|$. This means that for this subset of curves the problem of counting by discriminant and counting by naive height are essentially equivalent. They then require one of two assumptions. The cleaner of the two assumptions is that the quotient $\Delta(E)/C(E)$, which they call the index, is square-free. The second assumption, which is not disjoint from the first, is a bound on the so-called Szpiro ratio defined by 
	
	\begin{equation} \label{sratio} \beta_E = \frac{\log |\Delta(E)|}{\log C(E)}.
	\end{equation}
	
	Their second assumption is the requirement that $\beta_E \leq \kappa$ for some $\kappa < 7/4$. \\ 
	
	In our situation we do not require any restrictions on the size of the $j$-invariant. This is because unlike the large family case it is possible to count the number of elliptic curves with rational 2-torsion by discriminant precisely; see Theorem 1.9 in \cite{TX}. On the other hand some restrictions on the $p$-adic valuation of the discriminant akin to the assumptions in \cite{SSW} mentioned above are necessary. \\
	
	Crucial to our arguments is the existence of a canonical degree-$2$ isogeny for each curve $E_{a,b}$ in $\E_2$, defined by:
	\begin{equation} \label{2-isog} \phi : \E_2 \rightarrow \E_2, E_{a,b} \mapsto E_{-2a, a^2 - 4b}.
	\end{equation}
	The conductor of an elliptic curve is invariant under isogeny, so 
	\[C(E_{a,b}) = C(E_{-2a, a^2 - 4b}).\]
	However, the discriminant is not in general invariant under isogeny. It is thus more natural to consider the Szpiro ratios of $E_{a,b}$ and $E_{-2a, a^2 - 4b}$ simultaneously. \\
	
	We will require the notion of the \emph{conductor polynomial} for an elliptic curve $E_{a,b} \in \E_2$:

	\begin{equation} \label{condpoly} \C(E_{a,b}) = b(a^2 - 4b).
	\end{equation}
	We also define
	\begin{equation} \label{ind} \ind(E_{a,b}) = \frac{\C(E_{a,b})}{C(E_{a,b})}.
	\end{equation} 
	Note that $\ind(E_{a,b}) \in \bZ$. \\ \\
	Observe that 
	\begin{align*} \ind(\phi(E_{a,b})) & = \ind(E_{-2a, a^2 - 4b}) \\
	& = \frac{(a^2 - 4b)(4a^2 - 4a^2 + 4b)}{C(E_{a,b})} \\
	& = 4 \ind(E_{a,b}).
	\end{align*}	

	The analogues of the non-archimedean conditions in \cite{SSW} for our situation will be: 
	
	\begin{enumerate}
	\item Either the conductor polynomial $\C(E_{a,b})$ is cube-free, equivalently that $\ind(E)$ (as defined by (\ref{ind})) is square-free; or
	\item The average of the Szpiro ratios $\beta(E_{a,b})$ and $\beta(E_{-2a, a^2 - 4b})$ is less than $155/68$. 
	\end{enumerate} 
	
	The value of the number $155/68 > 9/4$ is significant because $9/4$ is the natural analogue of the constant $7/4$ as an upper bound for the Szpiro ratio in \cite{SSW}, obtained using geometry of numbers. Therefore the positivity of $\Psi = 155/68 - 9/4$ represents progressing beyond the simple application of geometry of numbers present in \cite{SSW}. \\
	
	To see that $9/4$ is the analogue of $7/4$ in our case, note that in any large family the average value of the Szpiro ratio is expected to be $1$. In the family $\E_2$, however, due to the presence of the square term the average value of the Szpiro ratio is expected to be $3/2$ instead. We then see that $7/4 - 1 = 3/4$ and $3/2 + 3/4 = 9/4$, justifying $9/4$ as the natural boundary of the methods of \cite{SSW} when applied in the present setting. \\
	
	We define, for $\kappa < \frac{155}{68} $, the family
	\begin{equation} \label{kappa fam} \E_{2,\kappa} = \{E \in \E_2 : (\beta_E + \beta_{\phi(E)})/2 \leq \kappa\}
 	\end{equation} 
	and 
	\begin{equation} \label{3-free} \E_{2}^\ast = \{E \in \E_2 : \C(E_{a,b}) \text{ is cube-free}\}.
	\end{equation}
	
	The first theorem in this paper is the following, which gives a an asymptotic formula for the number of curves in the families $\E_{2,\kappa}$ and $\E_2^\ast$, is the following: 
	\begin{theorem} \label{MT1} Let $1 < \kappa  < 155/68$ be a positive number. Then we have
	\[\# \{E \in \E_2^\ast : C(E) < X \} \sim \frac{1}{64} \cdot \frac{(2 +  \sqrt{2}) \Gamma(1/4)^2}{3 \sqrt{\pi}} \prod_{p \geq 5} \left( 1 - \frac{2p-1}{p^3} + \frac{2(p-1)^2}{p^{13/4}}   \right) X^{\frac{3}{4}},  \]
	\[\# \{E \in \E_{2,\kappa} : C(E) < X\} \sim \frac{1}{64} \cdot \frac{(2 +  \sqrt{2}) \Gamma(1/4)^2}{3 \sqrt{\pi}} \prod_{p \geq 5} \left(1 + \frac{1}{p^2} + p^{\frac{3}{2}} \frac{p-1}{p^4} + \frac{2(p-1)^2}{p^3(p^{1/4} - 1)} \right) X^{\frac{3}{4}} \]
	\[\# \{E_{a,b} \in \E_2 : |\C(E_{a,b})| < X\} \sim \frac{1}{64} \cdot  \frac{(2 +  \sqrt{2}) \Gamma(1/4)^2}{3 \sqrt{\pi}} \prod_{p \geq 5} \left( 1 - \frac{1}{p^6} \right) X^{\frac{3}{4}} .\]
	\end{theorem}
	
	Note that the constant $1/64$ in the statement of Theorem \ref{MT1} comes from the assumption that our curves have good reduction at $2$ and $3$, which contributes $1/32$, the remainder coming from the real density. This assumption is introduced to avoid having to deal with the conductor polynomial being divisible by arbitrarily large powers of $2,3$ while the conductor remains bounded at $2,3$. \\
	
	We expect $\E_{2, \kappa}$ to satisfy the asymptotic formula in Theorem \ref{MT1} for all $\kappa > 1$. The $abc$-conjecture implies that there are only finitely many elliptic curves with $\beta_E > 6$. This then shows that if we replace $\E_{2, \kappa}$ with $\E_2$ that the second asymptotic formula in (\ref{MT1}) will hold as well. The $p$-adic densities present in the asymptotic formulae arise from the densities of elliptic curves in $\E_2$ over $\bQ_p$ with fixed Kodaira symbol. These densities are computed in Section \ref{Kod symbols}. \\
	
	We compare Theorem \ref{MT1} with the main theorems in \cite{SSW}. The value $\kappa < 9/4$ obtained in \cite{SSW} is from fine tuning the geometry of numbers method pioneered by Bhargava in \cite{Bha1} and extended to the space of binary quartic forms by Bhargava and Shankar in \cite{BhaSha1}. If we apply only geometry of numbers methods in the present paper, we will obtain the analogous quantity $\kappa < 9/4$. To push beyond this barrier we will require an application of a theorem of Browning and Heath-Brown in \cite{B-HB}. This theorem relies on the $p$-adic determinant method of Heath-Brown \cite{HB1}, refined by Salberger in \cite{Sal1}. One also requires certain linear programming bounds in order to make good use of this method; this is similar to joint work of the author and B.~Nasserden in \cite{NX}. \\
	
	Our next theorem follows from adapting the methods in \cite{B-Ho} on counting 3-Selmer elements of curves in $\E_2$: 
	
	\begin{theorem} \label{MT2} When elliptic curves in $\E_2$ are ordered by their conductor polynomials, the average size of their 3-Selmer groups is bounded.
	\end{theorem} 
	
	We remark that Theorem \ref{MT2} is weaker than Theorem 1.2 in \cite{SSW} in that we do not have an exact count of the average, and we are only able to obtain boundedness for $\E_2$ when ordered by the conductor polynomial. The reason is that the parametrization of 3-Selmer elements of $\E_2$ obtained by Bhargava and Ho \cite{B-Ho} uses a much more complicated coregular space of very high dimension, and as such we are unable to apply the $p$-adic determinant method which at present only provides superior results in the case of curves and surfaces. 
	
	\subsection{Uniformity estimates} In order to prove Theorem \ref{MT1} we will need to prove certain tail estimates. Indeed, we will require the following theorems:
	
\begin{theorem}[Uniformity estimate for curves with cube-free conductor polynomial] \label{cubefree uniform} For all $\delta > 0$ there exists a positive number $\kappa$ such that 
\[\# \left\{E_{a,b} \in \E_2^\ast : C(E_{a,b}) \leq X, \ind(E_{a,b}) > X^{2 \delta} \right \} = O_{\delta, \kappa} \left(X^{\frac{3}{4} - \kappa} \right).\]
\end{theorem}

For curves for which $(\beta_E + \beta_{\phi(E)})/2$ is bounded, we have the following:

\begin{theorem}[Uniformity estimate for curves with bounded average Szpiro ratio] \label{szpiro uniform} Suppose that $1 < \kappa < 155/68$ and $\theta > 0$. Then there exists $\kappa^\prime$, depending only on $\kappa$ and $\theta$, such that
\begin{equation} \label{szpiroct} \# \left\{E_{a,b} \in \E_{2,\kappa} : C(E_{a,b}) \leq X, \frac{3}{2} + \theta < \frac{\beta_E + \beta_{\phi(E)}}{2} \leq \kappa \right\} = O_{\kappa^\prime} \left(X^{\frac{3}{4} - \kappa^\prime} \right). \end{equation} 
\end{theorem}

A summary of these uniformity estimates in the case of large families is given in \cite{SSW}; we refer the reader to the aforementioned paper for historical progress. As mentioned in \cite{SSW}, the main difficulty in proving Theorems \ref{cubefree uniform} and \ref{szpiro uniform} is that the size of the conductor polynomial can be very large for curves with bounded conductor. This necessitates some new ideas. \\

Departing from the ideas given in \cite{SSW}, our new input is that the shape of our conductor polynomial allows us to turn the counting problem into one about counting integer points on a family of quadrics over $\bP^2$ or counting over sublattices of $\bZ^2$ defined by congruence conditions. Depending on the size of the parameters involved the bounds are stronger from one interpretation over the other. To do this we rely on a uniform estimate, essentially sharp, of counting integer points having bounded coordinates due to Browning and Heath-Brown \cite{B-HB}. This is an application of the $p$-adic determinant method of Heath-Brown \cite{HB1}, refined by Salberger to the global determinant method; see \cite{Sal1} and \cite{X2} for a summary. 

\subsection{Outline of the paper}

In Section \ref{Kod symbols} we will characterize the possible Kodaira symbols of curves in $\E_2$ and compute their relative densities. In Section \ref{condpolysec} we compute the real density of curves in $\E_2$ and prove the third part of Theorem \ref{MT1}. In Section \ref{countingthms} we prove the first two parts of Theorem \ref{MT1} assuming the uniformity estimates given by Theorems \ref{cubefree uniform} and \ref{szpiro uniform}. In Sections \ref{cube-free} and \ref{large szpiro} we prove the aforementioned uniformity estimates. Finally, in Section \ref{3-sel} we prove Theorem \ref{MT2}. 

\subsection*{Acknowledgements} We thank C.~Tsang for discussions during the early development phase of this paper. We also thank the referee for extensive comments that lead to several major improvements to the paper.

	\section{Kodaira symbols for curves in $\E_2$}
	\label{Kod symbols} 
	
	The Kodaira symbols of curves in $\E_2$, for primes $p \geq 5$ based on the Table 1 of \cite{SSW}, can be significantly simplified. This is because our minimal Weierstrass model is already minimal with respect to every prime $p \geq 5$ because the constant coefficient is zero. In particular, we see at once that all of the symbols requiring a power of $p$ to exactly divide the constant coefficient and for $p$ to divide the $x^2$ and $x$-coefficients are not possible for curves in $\E_2$. This leaves only $\operatorname{I}_n, n \geq 1, \operatorname{I}_0^\ast, \operatorname{III}$, and $\operatorname{III}^\ast$ as possible Kodaira symbols in our family for $p \geq 5$. \\
	
	For $p = 2,3$ separate treatment is necessary. Recall that we defined our $\E_2$ to consist of curves with good reduction at $2,3$. Therefore, it suffices to take advantage of the work of Mulholland in \cite{Mul} to identify congruence conditions on $a,b$ which give good reduction at $2,3$. 
	
	\subsection{Good reduction at $2$ and $3$} We use the following statement, which follows from Theorems 2.1 and 2.3 in \cite{Mul}. 
	
	\begin{proposition} \label{23red} Suppose that $E_{a,b}$ is given as in (\ref{main fam}). 
	\begin{enumerate}
	\item $E_{a,b}$ has good reduction at $2$ if and only if either $b \equiv 1 \pmod{2}$ and $a \equiv 6 \pmod{8}$, or $a \equiv 1 \pmod{4}$, $b \equiv 16 \pmod{32}$.
	\item $E_{a,b}$ has good reduction at $3$ if and only if either $a \equiv 1,2 \pmod{3}$ and $b \equiv 2 \pmod{3}$, or $3 | a$ and $3 \nmid b$. 
	\end{enumerate}
	\end{proposition} 
	
	\begin{proof} See Tables 2.1 and 2.3 in \cite{Mul}. 
	\end{proof}
	
	We now proceed to compute the relative density of curves with good reduction. In the case of good reduction at $2$, the congruence conditions are determined by $(\bZ/32 \bZ)^2$. The first case in (1) of Proposition \ref{23red} contributes a density of $1/16$, and the second case contributes $1/128$, for a total of $9/128$. \\
	
	For the density of $(a,b)$ such that $E_{a,b}$ has good reduction at $3$, we note that the first case contributes $2/9$ to the density. The second case we must allow for the possibility that $a$ is divisible by an arbitrarily large power of $3$. Let $k \geq 1$ be given, and suppose $3^k$ exactly divides $a$. The number of pairs $(a,b) \in (\bZ/3^{k+1} \bZ)^2$ such that $a \equiv 3^k, 2 \cdot 3^k \pmod{3^{k+1}}$ and $3 \nmid b$ is $2(3^{k+1} - 3^k) = 4 \cdot 3^k$. Thus, the density is $8 \cdot 3^{k - 2k - 2} = 8 \cdot 3^{-k-2}$. Summing over $k$, we obtain a density of 
	\[\sum_{k=1}^\infty \frac{8}{3^{k+2}} = \frac{8}{27} \sum_{k=0}^\infty \frac{1}{3^k} = \frac{8}{27} \cdot \frac{1}{1-1/3} = \frac{4}{9}.  \]
	Combining the considerations of good reduction at $2$ and $3$, we conclude that the density of pairs $(a,b) \in \bZ^2$ such that $E_{a,b}$ has good reduction at $2$ and $3$ is $(9/128)(4/9) = 1/32$.

	\subsection{Contributions to the conductor for type-$\operatorname{I}_0^\ast$ primes}
	
	Put $a = p^k a_0$ and $b = p^\ell b_0$. Then our conductor polynomial takes the form 
	\[\C(E_{a,b}) = p^{ \ell} b_0 (p^{2k} a_0^2 - 4 p^{\ell} b_0).\]
	The constraint that $p^7 \nmid \Delta(E_{a,b})$ implies that 
	\[2 \ell + \min\{2k, \ell\} \leq 6.\]
	We also have $k \geq 1, \ell \geq 2$. If $\ell > 2$ then $\min\{2k, \ell\} > 2$, and therefore $2 \ell + \min\{2k, \ell\} > 6$. Hence we must have $\ell = 2$ and $k = 1$. This implies that $\C(E_{a,b})$ is exactly divisible by $p^4$. In this case we note that the conductor $C(E_{a,b})$ is only divisible by $p^2$. \\
	
	Therefore, the congruence information for a prime of Kodaira symbol $\KTI$ is contained in the $\bZ$-module $(\bZ/ p^3 \bZ)^2$. For $(a,b) \in (\bZ/p^3 \bZ)^2$, we have that $E_{a,b}(\bZ/p^3 \bZ)$ has Kodaira symbol $\KTI$  if and only if $a \equiv 0 \pmod{p}$ and $b \equiv 0 \pmod{p^2}, b \not \equiv 0 \pmod{p^3}$. Thus there are $p^2 (p-1)$ choices for $(a,b) \in (\bZ/p^3 \bZ)^2$, and the relative density is $(p-1)/p^4$.

	\subsection{Contributions to the conductor for type-$\operatorname{III}$ primes}
	
	If a given curve $E_{a,b} \in \E_2$ has Kodaira symbol $\operatorname{III}$ at a prime $p$, then we must have $p | b$ exactly. We may then put $b = pb_0$ where $p \nmid b_0$. Our curve then has the equation
	\[E_{a,b} : y^2 = x^3 + ax^2 + pb_0x.\] 
	If we put $a = p^k a_0$ with $p \nmid a_0$, then our conductor polynomial is equal to 
	\[\C(E_{a,b}) = p b_0 (p^{2k} a_0^2 - 4pb_0).  \]
	We then see that $p^2$ exactly divides the conductor polynomial, and hence exactly divides the conductor $C(E_{a,b})$. Therefore all of the congruence information is contained in the ring $(\bZ/p^2 \bZ)^2$, and we have that $E_{a,b}$ has Kodaira symbol $\KT$ at $p$ if and only if $a \equiv 0 \pmod{p}$ and $b \equiv 0 \pmod{p}, b \not \equiv 0 \pmod{p^2}$. There are $p(p-1)$ such possibilities, and the relative density is $(p-1)/p^3$. 
	
	\subsection{Contributions to the conductor for type-$\operatorname{III}^\ast$ primes}
	
	If a given curve $E_{a,b} \in \E_2$ has Kodaira symbol $\operatorname{III}^\ast$ at a prime $p$, then we must have $p^3 | b$ exactly. We may then put $b = p^3 b_0$ where $p \nmid b_0$. Our curve then has the equation
	\[E_{a,b} : y^2 = x^3 + ax^2 + p^3 b_0x.\] 
	If we put $a = p^k a_0$ with $p \nmid a_0$, then our conductor polynomial is equal to 
	\[\C(E_{a,b}) = p^3 b_0 (p^{2k} a_0^2 - 4p^3 b_0).  \]
	Since $k \geq 2$ by Table 1 of \cite{SSW}, it follows that $p^6$ exactly divides the conductor polynomial. Note that in this case the conductor $C(E_{a,b})$ is divisible by $p^2$. Thus the congruence condition is contained in the $\bZ$-module $(\bZ/p^4 \bZ)^2$, and $E_{a,b}(\bZ/p^4 \bZ)$ has Kodaira symbol $\KTS$ if and only if $a \equiv 0 \pmod{p^2}$ and $b \equiv 0 \pmod{p^3}, b \not \equiv 0 \pmod{p^4}$. Thus there are $p^2 (p-1)$ choices for $(a,b)$, and the relative density is $(p-1)/p^6$. 
	
	\subsection{Contribution to the conductor for semi-stable primes} 
	
	Recall that for an elliptic curve $E/\bQ$, $E$ has multiplicative bad reduction at $p$ if and only if $E$ has semi-stable bad reduction at $p$. Now the exponent $k \geq 1$ can be arbitrarily large. Over the $\bZ$-module $(\bZ/p^{k+1} \bZ)^2$ a pair $(a,b) \in (\bZ/p^{k+1} \bZ)^2$ corresponds to an elliptic curve $E_{a,b}$ having semi-stable bad reduction at $p$ if and only if $p$ divides exactly one of $b$ and $c = a^2 - 4b$. If $b \equiv 0 \pmod{p^k}, b \not \equiv 0 \pmod{p^{k+1}}$ then we must have $a \not \equiv 0 \pmod{p}$. Thus there are $p^{k}(p-1)$ choices for $a$ and $p-1$ choices for $b$. If $c \equiv 0 \pmod{p^{k}}, c \not \equiv 0 \pmod{p^{k+1}}$ then we cannot have $a \equiv 0 \pmod{p}$, since otherwise $b \equiv 0 \pmod{p}$ also which is not allowed. Therefore for any $a$ co-prime to $p$ we may choose a unique $b \pmod{p^{k+1}}$ such that $c \equiv 0 \pmod{p^k}$ and $c \not \equiv 0 \pmod{p^{k+1}}$. There are again $p^{k}(p-1)^2$ such choices. Moreover it is clear that the two sets of possibilities are disjoint. It follows that there are $2p^{k}(p-1)^2$ possibilities, and the relative density is $2(p-1)^2/p^{k+2}$. 

\section{The family $\E_2$ ordered by conductor polynomial}
\label{condpolysec} 

Recall that our family $\E$ is given by the equation (\ref{main fam}). Analogous to our construction \cite{TX}, we introduce the \emph{conductor polynomial} of $E_{a,b}$ as 
	\begin{equation} \label{condpoly} \C(E_{a,b}) = b(a^2 - 4b), c = a^2 - 4b. \end{equation}

Let $A_\infty(Z)$ be the Lebesgue measure of the set 
\[\A(Z) = \{(x,y) \in \bR^2 : |y(x^2 - y)| \leq Z\}.\]
The region $\A(Z)$ has four long cuspidal regions, defined by $y = 0$ and $x^2 = y$ respectively. Nevertheless, we compare the areas of $\A(Z)$ and $\A^\ast(Z)$ defined below: 
\begin{equation} \label{area} \A^\ast(Z) = \{(x,y) \in \bR^2 : |y(x^2 - y)| \leq Z, |x^2 - y|, |y| \geq 4\}
\end{equation}
and show that they have comparable areas for $Z$ large. We note that the set $\A^\ast(Z)$ is in fact bounded. To see this, note that $|x^2 - y|, |y| \geq 4$ implies that $|y|, |x^2 - y| \leq Z/4$. Hence $\A(Z)$ is contained in the region $\{(x,y) \in \bR^2 : |y|, |x^2 - y| \leq Z/4\}$ which is clearly bounded. \\

\begin{lemma} \label{Acomp} For a Lebesgue measurable set $\X \subset \bR^2$, put $m(\X)$ for its Lebesgue measure. We then have
\[m(\A(Z) \setminus \A^\ast(Z)) = O \left(Z^{\frac{1}{2}} \right). \]
\end{lemma} 

\begin{proof} See \cite{TX}, Proposition 3.12. 
\end{proof} 

We put $A_\infty(Z) = m(\A(Z))$. By Lemma \ref{Acomp}, the area of the bounded region $\A^\ast(Z)$ is comparable to $A_\infty(Z)$, so it suffices to compute $A_\infty(Z)$ which scales homogeneously with the parameter $Z$. In particular, we have 
\[A_\infty(Z) = A_\infty(1) Z^{\frac{3}{4}}.\]

To compute $A_\infty(Z)$, we note the symmetry about the $y$-axis and thus assume without loss of generality that $x > 0$. Call the corresponding region $\A_+(Z)$. We then partition $\A_+(Z)$ into three regions:
\begin{align*} \A_1 & = \left \{(x,y) \in \A^\ast(Z) : |x| \leq \sqrt{2} Z^{\frac{1}{4}} \right\}, \\
\A_2 & = \left\{(x,y) \in \A^\ast(Z) : |x| > \sqrt{2}Z^{\frac{1}{4}} , y >  \sqrt{Z}\right\}, \\
\A_3 & = \left\{(x,y) \in \A^\ast(Z): |x| > \sqrt{2} Z^{\frac{1}{4}} , y \leq \sqrt{Z} \right\}.
\end{align*}
Clearly, $\A_1, \A_2, \A_3$ partition $\A_+(Z)$. \\

We now compute $m(\A_1)$. To do this, for fixed $0 < x < \sqrt{2}Z^{\frac{1}{4}}$ we have $y$ lies in the interval defined by the equation $y(x^2 - y) = - Z$, giving 

\[y \in \left[ \frac{x^2 - \sqrt{x^4 + 4Z}}{2}, \frac{x^2 + \sqrt{x^4 + 4Z}}{2} \right]. \]
it follows that 
\begin{equation} m(\A_1) = \int_{0}^{\sqrt{2}Z^{\frac{1}{4}}}  \int_{\frac{x^2 - \sqrt{x^4 + 4Z}}{2}}^{\frac{x^2 + \sqrt{x^4 + 4Z}}{2}} dy dx = \int_{0}^{\sqrt{2}Z^{\frac{1}{4}}} \sqrt{x^4 + 4Z} dx.
\end{equation}
As in \cite{TX}, we will compute the above area using elliptic integrals. Make the substitution $x = \sqrt{2} Z^{\frac{1}{4}} z$ and $dx = \sqrt{2} Z^{\frac{1}{4}} dz$ gives 
\[m(\A_1) = 2 \sqrt{2} Z^{\frac{3}{4}} \int_0^1 \sqrt{z^4 + 1} dz = \frac{2 \sqrt{2} Z^{\frac{3}{4}}}{3} \left(\sqrt{2} + \frac{\Gamma(1/4)^2}{4 \sqrt{\pi}} \right).\]
Next we compute $m(\A_2)$. Similar to the considerations above we find that 
\[m(\A_2) = \int_{\sqrt{2} Z^{\frac{1}{4}}}^\infty \int_{\frac{x^2 - \sqrt{x^4 - 4Z}}{2}}^{\frac{x^2 - \sqrt{x^4 + 4Z}}{2}} dy dx = \frac{1}{2} \int_{\sqrt{2} Z^{\frac{1}{4}}}^\infty \left(\sqrt{x^4 + 4Z} - \sqrt{x^4 - 4Z} \right)dz.\]
Again making the substitution $x = \sqrt{2} Z^{\frac{1}{4}} z, dx = \sqrt{2} Z^{\frac{1}{4}} dz$ we find that 
\[m(\A_2) =  \sqrt{2} Z^{\frac{3}{4}} \int_1^\infty \left(\sqrt{z^4 + 1} - \sqrt{z^4 - 1} \right) dz.\]
To evaluate this integral we apply integration by parts and obtain
\begin{align*} \int_1^{\infty} \left(\sqrt{z^4 + 1} - \sqrt{z^4 - 1} \right)dz & = \frac{1}{3} \left[z \sqrt{z^4 + 1} - z \sqrt{z^4 - 1} \right]_1^\infty + \frac{2}{3} \int_1^\infty \frac{dz}{\sqrt{z^4 + 1}} + \frac{2}{3} \int_1^\infty \frac{dz}{\sqrt{z^4 - 1}} \\
& =  -  \frac{ \sqrt{2}}{3} + \frac{1}{3} \left(\frac{(1 + \sqrt{2}) \Gamma(1/4)^2}{4 \sqrt{\pi}} \right)  \\
& = \frac{1}{3} \left(-\sqrt{2} + \frac{(1 + \sqrt{2}) \Gamma(1/4)^2}{4 \sqrt{\pi}} \right).
\end{align*} 
This implies that 
\[m(\A_2) = \frac{ \sqrt{2} Z^{\frac{3}{4}}}{3}  \left(-\sqrt{2} + \frac{(1 + \sqrt{2}) \Gamma(1/4)^2}{4 \sqrt{\pi}} \right).\]
A similar calculation reveal that 
\[m(\A_2) = m(\A_3) =\frac{ \sqrt{2} Z^{\frac{3}{4}}}{3} \left(-\sqrt{2} + \frac{(1 + \sqrt{2}) \Gamma(1/4)^2}{4 \sqrt{\pi}} \right). \]
It follows that 
\begin{align*} m(\A_+^\ast(Z)) & = m(\A_1) + m(\A_2) + m(\A_3) \\
& = \frac{2\sqrt{2}Z^{\frac{3}{4}}}{3} \left(\sqrt{2} + \frac{\Gamma(1/4)^2}{4 \sqrt{\pi}} - \sqrt{2} + \frac{(1 + \sqrt{2}) \Gamma(1/4)^2}{4 \sqrt{\pi}} \right) \\
& = \frac{(1 + \sqrt{2}) \Gamma(1/4)^2}{3 \sqrt{\pi}}
\end{align*}
This gives
\begin{equation} A_\infty(Z) = \frac{2( 1+  \sqrt{2}) \Gamma(1/4)^2}{3 \sqrt{\pi}} Z^{\frac{3}{4}} + O \left(Z^{\frac{1}{2}} \right).
\end{equation}
	
We may use a refined version of Davenport's lemma, due to Barroero and Widmer \cite{BW}. Before we state the proposition, we introduce some notation in \cite{BW}. Here we consider, as in \cite{BW}, a parametrized family $\Z \subset \bR^{m+n}$ of subsets $\Z_T = \{\Bx \in \bR^n: (T, \Bx) \in \Z\}$. We consider a lattice $\Lambda \subset \bR^n$. The goal of the following proposition is to estimate the cardinality $|\Lambda \cap \Z_T|$ as the parameter $T$ ranges over an infinite set. They obtained the following, the form of which was stated in \cite{SSW}: 

\begin{proposition} \label{dav-BW} Let $m$ and $n$ be positive integers and let $\Lambda \subset \bR^n$ be a lattice. Denote the successive minima of $\Lambda$ by $\lambda_i, i = 1, \cdots, n$. Let $\Z \subset \bR^n$ be a definable family in an $o$-minimal structure, and suppose the fibres $\Z_T$ are bounded. Then there exists a positive number $c_\Z$, depending only on $\Z$, such that 
\[\left \lvert \# (\Z_T \cap \Lambda) - \frac{\Vol(\Z_T)}{\det(\Lambda)} \right \rvert \leq c_\Z \sum_{j=0}^{n-1} \frac{V_j(\Z_T)}{\lambda_1 \cdots \lambda_j} \]
where $V_j(\Z_T)$ is the sum of the $j$-dimensional volumes of the orthogonal projections of $\Z_T$ onto every $j$-dimensional coordinate subspace of $\bR^n$. 
\end{proposition} 

Suppose now that we are given a set $\S \subset \bZ^2$ defined by congruence conditions modulo some integer $n > 0$. Then we may break $\S$ up into a union of $n^2 \nu(\S)$ translates of the lattice $n \bZ \times n \bZ$, where $\nu(\S)$ denotes the volume of the closure of $\S$ in $\hat{\bZ}^2$. Applying Proposition \ref{dav-BW} to each of these translates and summing gives us the following result: 

\begin{proposition} \label{fincong} Let $\S \subset \bZ^2$ be a set of pairs $(a,b)$ defined by congruence conditions on $a,b$ modulo some positive integer $n$. Then we have 
\[\#\{(a,b) \in \S : |\C(E_{a,b})| \leq X, |b|, |a^2 - 4b| \geq 4\} = \frac{\sqrt{2}}{ 2} \nu(\S) A_\infty(X) + O_\ep \left(n \nu(\S) X^{\frac{1}{2} + \ep} \right).\]
\end{proposition}

\begin{proof} This is a direct consequence of Proposition \ref{dav-BW} and the observation that to count the correct points in $\A(X)$ we need to impose the condition $y \equiv 0 \pmod{4}$, which leads to the leading term $ \sqrt{2}/2$ in the front. See also Proposition 3.12 in \cite{TX}. 
\end{proof}

We now prove the third part of Theorem \ref{MT1}. Let us put $N_m(X)$ be the number of curves $E_{a,b} \in \E_2$ in $\A(X)$ such that $p^2 | a$ and $p^3 | b$ for each $p$ dividing $m$. It follows that 
\begin{align*} & \# \{E_{a,b} : |\C(E_{a,b})| \leq X, |b|, |a^2 - 4b| \geq 4\}  = \sum_{m \geq 1} \mu(m) N_m(X) \\
& = \sum_{m=1}^{X^\delta} \mu(m) \frac{A_\infty(X)}{8 \sqrt{2}} + \sum_{m=1}^{X^\delta} O \left(m X^{1/2 + \ep} \right) + O \left(\sum_{\substack{m \geq X^\delta \\  m \text{ square-free}}} N_m(X) \right) \\
& = \frac{A_\infty(X)}{8 \sqrt{2}} \prod_{p} \left(1 - \frac{1}{p^6} \right) + O \left(X^{\frac{1}{2} + 2 \delta + \ep} \right) + O \left( \sum_{X^\delta \leq m \leq X^{1/6}} \frac{X^{3/4 + \ep}}{m^6} \right). 
\end{align*}
Here the last estimate comes from the fact that by using the isogeny $\phi$ we may assume that $|b| \leq X^{1/2}$, and hence $m^3$ divides $|b|$ implies that $|m| \leq X^{1/6}$. We can optimize $\delta$ by solving
\[\frac{1}{2} + 2 \delta = \frac{3}{4} - 6 \delta \]
which gives 
\[\delta = \frac{1}{32}.\]
This is sufficient for the proof of the third part of Theorem \ref{MT1}.

\section{Main counting theorems assuming uniformity estimates} 
\label{countingthms} 

We follow the same approach as in \cite{SSW}, and note that our uniformity estimates given by Theorem \ref{cubefree uniform} imply 
\begin{align} \label{SSW44} \#\{E \in \E_2^\ast : C(E) < X \} & = \sum_{n \geq 1} \# \{E \in \E_2^\ast : \ind(E) = n, \C(E) < nX\} \\
& = \sum_{n,q \geq 1} \mu(q) \# \{E \in \E_2^\ast : nq | \ind(E), \C(E) < nX\} \notag \\
& = \sum_{\substack{n, q \geq 1 \\ nq < X^\delta}} \mu(q) \# \{E \in \E_2^\ast : nq | \ind(E), \C(E) < nX\} + O \left(X^{\frac{3}{4} - \delta^\prime} \right). \notag
\end{align}
The last line follows from our uniformity estimates. \\

We introduce some notation, following \cite{SSW}. For each prime $p \geq 5$ let $\Sigma_p$ be a non-empty subset of possible reduction types. We say that $\Sigma = (\Sigma_p)_p$ is a \emph{collection of reduction types}. We will assume that each $\Sigma_p$ contains the good and multiplicative reduction types for all $p \geq 5$. \\

We then perform another inclusion-exclusion sieve to evaluate each summand on the right-hand side of the expression above. For each prime $p$ let 
\[\chi_{\Sigma_p, nq} : \bZ_p^2 \rightarrow \bR\] 
be the characteristic function of the set of all $(a,b) \in \bZ_p^2$ that satisfy the reduction type specified by $\Sigma_p$ and satisfy $nq | \ind(E_{a,b})$. Let us put $ \chi_p = 1 - \chi_{\Sigma_p, nq}$ and define 
\[\chi_k := \prod_{p \mid k} \chi_p\]
for square-free integers $k$. Then we have
\begin{equation} \label{chip} \prod_p \chi_{\Sigma_p, nq}(a,b) = \sum_k \mu(k) \chi_k(a,b)
\end{equation}
for every $(a,b) \in \bZ^2$. Put $\nu_\ast(nq, \Sigma)$ to be the product over all primes $p$ of the integral of $\chi_{\Sigma_p, nq}$ over $\bZ_p^2$. Then for $nq < X^\delta$ we have 
\begin{align*} \# \{E \in \E_2^\ast : nq \mid \ind(E), \C(E) < nX\} & = \sum_{\substack{(a,b) \in \bZ^2 \\ 0 < |\C(E_{a,b})| < nX}} \sum_{k \geq 1} \mu(k) \chi_k(a,b) \\
& = \sum_{\substack{(a,b) \in \bZ^2 \\ 0 < |\C(E_{a,b})| < nX}} \sum_{k=1}^{X^{4\delta}} \mu(k) \chi_k(a,b) + O \left(X^{\frac{3}{4} - \kappa} \right) \\
& = A_\infty(nX) \nu_\ast(nq, \Sigma) + O_\ep \left(X^{\frac{1}{2} + \ep} + X^{\frac{3}{4} - \kappa} \right)
\end{align*}
where $\nu_\ast(nq, \Sigma)$ is the product over all primes $p$ of the $p$-adic integral of $\chi_{\Sigma_p, nq}$. For each $n$, put $\lambda_i(n, \Sigma)$ for the volume of the closure in $\hat{\bZ}^2$ of the set of all $(a,b) \in \bZ^2$ such that $E_{a,b}$ belongs to $\G = \E_2(\Sigma)$ and $E_{a,b}$ has index $n$. Returning to (\ref{SSW44}), we obtain
\begin{align*} \# \{E \in \G : C(E) < X\} & = A_\infty(1) X^{\frac{3}{4}} \sum_{\substack{n,q \geq 1 \\ nq < X^\delta}} \mu(q) n^{\frac{3}{4}} \nu_\ast(nq, \Sigma) + o \left(X^{\frac{3}{4}} \right) \\
& = A_\infty(1) X^{\frac{3}{4}} \sum_{n \geq 1} n^{\frac{3}{4}} \lambda_\ast(n, \Sigma),
\end{align*}
where the final equality follows by reversing the inclusion-exclusion sieve in (\ref{SSW44}) (see (44) in \cite{SSW}). \\

For each prime $p \geq 5$ and integer $k \geq 0$, put $\ol{\nu}_\ast(p^k, \Sigma)$ for the $p$-adic density of the set of all $(a,b) \in \bZ^2$ such that $E_{a,b} \in \E_\ast(\Sigma)$ and $\ind_p(E_{a,b}) = p^k$. The constant $\lambda_\ast(n, \Sigma)$ is a product over all $p$ of local densities:
\begin{align*} \lambda_\ast(n, \Sigma) & = \prod_{\substack{ p \geq 5 \\ p \nmid n}} \ol{\nu}_\ast(p^0, \Sigma) \prod_{\substack{p^k \vert \vert n \\ k \geq 1}} \ol{\nu}_\ast(p^k, \Sigma) \\
& = \prod_{p \geq 5} \ol{\nu}_\ast(p^0, \Sigma) \prod_{\substack{p^k \vert \vert n \\ k \geq 1}} \frac{\ol{\nu}_\ast(p^k, \Sigma)}{\ol{\nu}_\ast(p^0, \Sigma)} 
\end{align*}
It follows that $\lambda_\ast(n, \Sigma)$ is a multiplicative function of $n$, and hence 
\begin{align*} \sum_{n \geq 1} n^{\frac{3}{4}} \lambda_\ast(n, \Sigma) & = \prod_{p \geq 5} \ol{\nu}_\ast(p^0, \Sigma) \prod_p \left(\sum_{k=0}^\infty p^{\frac{3k}{4}} \frac{\ol{\nu}_\ast (p^k, \Sigma)}{\ol{\nu}_\ast(p^0, \Sigma)} \right) \\
& = \prod_{p \geq 5} \left(\sum_{k=0}^\infty p^{\frac{3k}{4}} \ol{\nu}_\ast (p^k, \Sigma) \right).
\end{align*}

The computation of $\ol{\nu}_\ast(p^k)$ then follows from the calculations in Section \ref{Kod symbols}. In the case of $\E_2^\ast$, we have that $p \mid \ind(E_{a,b})$ if and only if $p$ is semi-stable. Further, our imposition implies that $\ind(E_{a,b})$ is square-free in this case, so $k \leq 1$. We thus obtain the density 
\begin{equation} \frac{(p-1)^2}{p^2} + \frac{p-1}{p^3} + \frac{2(p-1)^2}{p^3} + p^{\frac{3}{4}} \frac{2(p-1)^2}{p^4} = 1 - \frac{2p-1}{p^3} + \frac{2(p-1)^2}{p^{13/4}} 
\end{equation} 
For the general case, we have that the bulk of the contribution comes from the semi-stable primes. The density is then seen to be
\begin{align} \sum_{k=1}^\infty p^{\frac{3(k-1)}{4}} \frac{2(p-1)^2}{p^{k+2}} & = \frac{2(p-1)^2}{p^3} + \sum_{k=2}^\infty p^{\frac{3(k-1)}{4}} \frac{2(p-1)^2}{p^{k+2}} \\ \notag
& =  \frac{2(p-1)^2}{p^3} + \frac{2(p-1)^2p^{\frac{3}{4}}}{p^4} \sum_{k=0}^\infty p^{\frac{3k}{4}} p^{-k} \\ \notag 
& = \frac{2(p-1)^2}{p^3} + \frac{2(p-1)^2}{p^{13/4}} \frac{p^{1/4}}{p^{1/4} - 1} \\ \notag
& = \frac{2(p-1)^2}{p^3} + \frac{2(p-1)^2}{p^3(p^{1/4} - 1)}. \notag
\end{align} 
The contribution from the other Kodaira symbols occurs for $k = 0, 2, 4$. Combined with the contribution. This gives the total contribution 
\begin{align} & \frac{(p-1)^2}{p^2} + \frac{p-1}{p^3} + p^{\frac{3}{2}} \frac{p-1}{p^4} + p^3 \frac{p-1}{p^6} + \frac{2(p-1)^2}{p^3} + \frac{2(p-1)^2}{p^3(p^{1/4} - 1)} \\ \notag
& = 1 - \frac{2p-1}{p^2} + \frac{p-1}{p^3} + \frac{p-1}{p^3} + p^{\frac{3}{2}} \frac{p-1}{p^4} + \frac{2p^2 - 4p + 2}{p^3} + \frac{2(p-1)^2}{p^3(p^{1/4} - 1)} \\ \notag
& = 1 - \frac{2p^2 - p - 2p + 2  - 2p^2 + 4p - 2}{p^3} + p^{\frac{3}{2}} \frac{p-1}{p^4} + \frac{2(p-1)^2}{p^3(p^{1/4} - 1)} \\ \notag
& = 1 + \frac{1}{p^2} + p^{\frac{3}{2}} \frac{p-1}{p^4} + \frac{2(p-1)^2}{p^3(p^{1/4} - 1)}. \notag \end{align}
This suffices for the proof of Theorem \ref{MT1}.

	\section{Counting curves with cube-free conductor polynomial}
	\label{cube-free}
	
	In this section we prove the necessary uniformity estimates in the sieve given in Section \ref{countingthms} in the case when the conductor polynomial $\C(E_{a,b})$ is cube-free. \\
		
	In order to take advantage of these results,  we note that $\C(E_{a,b})$ being cube-free implies the curve $E_{a,b}$ has no primes of Kodaira symbol $\KTI$ and $\KTS$. Let $P = p_1 \cdots p_k$ be the product of a finite number of primes, and we shall restrict our attention to those curves $E \in \E_2^\ast$ which have Kodaira symbol $\KT$ at each of the primes dividing $P$, and no other primes. We then have that $a = Pu, b = Pv$ for some $u,v \in \bZ$ and
\[\C(E_{a,b}) = P^2 v\left(Pu^2 - 4v\right)\]
Further, the integers $v$ and $w = Pu^2 - 4v$ are co-prime to $P$. Our cube-free condition then implies we that we may express $v, w$ in the form: 

\begin{equation} \label{sqf1} v = v_0 v_1^2, \gcd(v_0, v_1) = 1, v_0, v_1 \text{ square-free}\end{equation}
and
\begin{equation} \label{sqf2} w = Pu^2 - 4v = w_0 w_1^2, \gcd(w_0, w_1) = 1, w_0, w_1 \text{ square-free}.\end{equation} 
Thus we obtain a quadratic curve over $\bP^2$ defined by the equation
\begin{equation} \label{quad form} Pu^2 = w_0 w_1^2 + 4v_0 v_1^2.
\end{equation}
We then see that the conductor is equal to
\[C(E_{a,b}) = P^2 |v_0 v_1 w_0 w_1|\]
and
\begin{equation} \label{cubefree index} \ind(E_{a,b}) = \frac{P^2 |v_0 v_1^2 w_0 w_1^2|}{P^2 |v_0 v_1 w_0 w_1|} = |v_1 w_1|.
\end{equation}
We now prove Theorem \ref{cubefree uniform}. The proof will follow from Theorem \ref{uniform 1} below. \\

We give some further preliminaries before stating Theorem \ref{uniform 1}. Our bound on the conductor is equivalent to 
\begin{equation} \label{quad bound} |v_0 v_1 w_0 w_1| \leq XP^{-2} = Z,
\end{equation}
say. We then restrict $v_0, v_1, v_0, v_1$ into dyadic boxes 
\begin{equation} \label{dya range} |v_0| \asymp T_1, |w_0| \asymp T_2, |v_1| \asymp T_3, |w_1| \asymp T_4.\end{equation} 
Put
\begin{equation} \label{expdef} t_i = \frac{\log T_i}{\log Z},
\end{equation}
and by further dividing into dyadic ranges and considering (\ref{quad bound}), we may suppose that
\begin{equation} \label{ti sum}  t_1 + t_2 + t_3 + t_4 = 1.\end{equation} 
We will be concerned with estimating
\begin{equation} \label{dya box} R(T_1, T_2, T_3, T_4) =  \{E_{a,b} : y^2 = x(x^2+ax +b), (\ref{sqf1}) \text{ to } (\ref{expdef}) \text{ hold} \}. 
\end{equation}
	
We shall prove the following result, which implies Theorem \ref{cubefree uniform}: 
\begin{theorem} \label{uniform 1} Suppose $t_i, i = 1, 2, 3,4$ are given as in (\ref{expdef}), satisfy (\ref{ti sum}), and $4 \delta < t_1 + t_2 < 1 - 4 \delta$. Then there exists $\kappa > 0$ such that 
\[R(T_1, T_2, T_3, T_4) = O_\kappa \left(Z^{\frac{3}{4} - \kappa} \right).\]
\end{theorem}

	 The proof requires two lemmata which counts points inside sub-lattices of $\bZ^2$ and points of bounded height on quadric curves in $\bP^2$ effectively. In particular we require the following result due to Browning and Heath-Brown \cite{B-HB}. It provides an essentially sharp upper bound for the number of rational points of bounded height on a quadric curve in $\bP^2$. 
\begin{proposition}[Corollary 2, \cite{B-HB}] \label{BHBcor} Let $\Q \in \bZ[x_1, x_2, x_3]$ be a non-singular quadratic form with matrix $M$. Let $\Delta = |\det(M)|$ and write $\Delta_0$ for the greatest common divisor of the $2 \times 2$ minors of $M$. Then we have
\begin{equation} \# \{(x_1, x_2, x_3) \in \bZ^3 : \gcd(x_1, x_2, x_3) = 1, |x_i| \leq R_i \text{ for } i = 1,2,3, \Q(x_1, x_2, x_3) = 0\} 
\end{equation}
\[\ll \left\{1 + \left(\frac{R_1 R_2 R_3 \Delta_0^{3/2}}{\Delta} \right)^{\frac{1}{3}} \right\} \tau(\Delta)\]
where $\tau(\cdot)$ is the divisor function. 
\end{proposition} 	
The strength of this Proposition lies in its uniformity with respect to the height of the coefficients of $\Q$, in that it is entirely independent of the height except on the determinant of $M$. A similar result, due to Heath-Brown \cite{HB3}, holds for sub-lattices of $\bZ^2$.
\begin{lemma} \label{lat lem} Let $\Lambda \subset \bZ^2$ be a lattice. Then for all positive real numbers $R_1, R_2$ the number of primitive integral points $\Bx \in \Lambda$ satisfying $|x_i| \leq R_i, i = 1,2$ is at most $O \left(\dfrac{R_1 R_2}{\det(\Lambda)} + 1 \right)$. 
\end{lemma}
\begin{proof} See Lemma 2 in \cite{HB3}. 
\end{proof}

Now to prove Theorem \ref{uniform 1}, we will consider various conditions on the $t_i$'s each of which gives satisfactory bounds for $R(T_1, T_2, T_3, T_4)$, and together cover all cases except for $t_1 + t_2 \leq 1 - 2\delta$. We first prove the following proposition:
\begin{proposition} Suppose that there exists $\delta_1 > 0$ such that either 
\begin{equation} \label{3s1} t_1 + t_2 + t_3 \leq \frac{3}{4} - \delta_1 \text{ or } t_1 + t_2 + t_4 \leq \frac{3}{4} - \delta_1.\end{equation} 
Then there exists $\kappa > 0$ such that 
\[R(T_1, T_2, T_3, T_4) = O_\kappa \left(X^{\frac{3}{4} - \kappa} \right).\] 
\end{proposition}
\begin{proof} Without loss of generality, we may suppose that 
\[t_1 + t_2 + t_3 \leq \frac{3}{4} - \delta_1.\]
Then we may choose $v_0, v_1, w_0$ in $O(T_1 T_2 T_3) = O \left(X^{3/4 - \delta_1}\right)$ ways. Having done so, (\ref{quad form}) becomes 
\[Pu^2 - w_0 w_1^2 = 4v_0 v_1^2\]
where $w_0, v_0, v_1$ are fixed and $u, w_1$ are bounded by a power of $X$. The left hand side is a quadratic form, and thus this equation has $O(\tau(4v_0 v_1^2) \log X) = O_\ep \left(X^{\ep}\right)$ many solutions. Choosing $\ep = \delta_1/2$ we conclude that
\[R(T_1, T_2, T_3, T_4) = O_\delta \left(X^{\frac{3}{4} - \frac{\delta_1}{2}} \right),\]
and choosing $\kappa = 2\delta_1$ say gives us the desired result. 
\end{proof} 

We now consider the quadruples $(t_1, t_2, t_3, t_4)$ for which (\ref{3s1}) fails. In fact we consider a slightly larger set of quadruples. In particular, we choose some $\delta_2 > 0$, arbitrarily small, so that 
\begin{equation} \label{quad fail} t_1 + t_2 + t_3 \geq \frac{3}{4} - \delta_2 \text{ and } t_1 + t_2 + t_4 \geq \frac{3}{4} - \delta_2 .\end{equation}

We now show that fixing $v_0, w_1$ and treating (\ref{quad form}) as counting integral points in $O_\ep(X^\ep)$ lattices of the shape 
\[\{(x,y) \in \bZ^2 : x \equiv \omega y \pmod{4v_1^2}, |x| \asymp T_1^{1/2} T_3, |y| \asymp T_4\}\]
will allow us to recover $u = x, w_1 = y$ which then gives $v_0 = (Pu^2 - w_0 w_1^2)/4v_1^2$. Applying Lemma \ref{lat lem} gives the bound 
\[O \left(\frac{T_1^{1/2} T_3 T_4}{T_3^2} + 1 \right) = O \left(\frac{T_1^{1/2} T_4}{T_3} + 1\right).
\]
Summing over $|w_0| \asymp T_2, |v_1| \asymp T_3$ gives
\begin{equation} O \left(T_1^{1/2} T_2 T_4 + T_2 T_3 \right).
\end{equation} 
Symmetrically, we obtain the bound 
\[O \left(T_2^{1/2} T_1 T_3 + T_1 T_4 \right).\]
This provides a satisfactory bound if for some $\delta_2 > 0$ we have
\begin{equation} \label{ineq1} \frac{t_1}{2} + t_2 + t_4 < \frac{3}{4} - \delta_2 \text{ and } t_2 + t_3 < \frac{3}{4} - \delta_2 \end{equation} 
or
\begin{equation} \label{ineq2} \frac{t_2}{2} + t_1 + t_3 < \frac{3}{4} -\delta_2 \text{ and } t_1 + t_4 < \frac{3}{4} - \delta_2.\end{equation}
We summarize this as:
\begin{proposition} Suppose that $t_1, t_2, t_3, t_4$ satisfies (\ref{quad fail}) and one of (\ref{ineq1}) or (\ref{ineq2}). Then 
\[R(T_1, T_2, T_3, T_4) = O_{\delta_2, \ep} \left(X^{3/4 - \delta_2 + \ep}\right).\]
\end{proposition}
Next let us analyze how both (\ref{ineq1}) and (\ref{ineq2}) can fail for every $\delta_2 > 0$. First note that it is not possible for 
\[t_2 + t_3 \geq \frac{3}{4} - \delta_2  \text{ and } t_1 + t_4 \geq \frac{3}{4} - \delta_2  \]
if $\delta_2$ is sufficiently small, since $t_1 + t_2 + t_3 + t_4 = 1$ by (\ref{ti sum}). Therefore, upon assuming $\delta_2$ is arbitrarily small, we may suppose that 
\[\frac{t_1}{2} + t_2 + t_4 \geq \frac{3}{4} - \delta_2 \text{ and } \frac{t_2}{2} + t_1 + t_3 \geq \frac{3}{4} - \delta_2 \]
or
\[\frac{t_1}{2} + t_2 + t_4 \geq \frac{3}{4} - \delta_2  \text{ and } t_1 + t_4 \geq \frac{3}{4} - \delta_2 ,\]
and a symmetric case which is equivalent to the line above. In the first case we obtain by summing the two inequalities
\[\frac{3}{2} (t_1 + t_2) + t_3 + t_4 \geq \frac{3}{2} - 2 \delta_2,\]
which implies that
\[t_1 + t_2 \geq 1 - 4 \delta_2.\]
This then implies $t_3 + t_4 \leq 4 \delta_2$, so are excluded from the theorem as long as $\delta_2 < \delta$. In the second case we have 
\begin{equation} \label{4s1} \frac{t_1}{2} + t_2 + t_4  \geq \frac{3}{4} - \delta_2  \text{ and } t_1 + t_4 \geq \frac{3}{4} - \delta_2. \end{equation}
By (\ref{ti sum}), the latter inequality implies 
\[t_2 + t_3 \leq \frac{1}{4} + \delta_2 ,\]
and therefore (\ref{quad fail}) gives
\[t_1 \geq \frac{1}{2} - 2 \delta_2.\]
We note that the condition
\begin{equation} \label{t2t4 cond} \frac{t_1}{2} + t_2 + t_4 \geq \frac{3}{4} - \delta_2 \end{equation}
and (\ref{ti sum}) imply that $t_1 \leq 1/2 + 2 \delta_2$. To see this, by (\ref{ti sum}) we have
\[t_2 + t_4 \leq 1  - t_1.\]
If $t_1 > 1/2 + 2 \delta_2$, then in turn we find that 
\begin{align*} \frac{t_1}{2} + t_2 + t_4 & \leq \frac{t_1}{2} + 1  - t_1 \\
& = 1  - \frac{t_1}{2} < 1  - \frac{1}{4} - \delta_2   \\
& = \frac{3}{4}- \delta_2 .
\end{align*}
This in turn violates (\ref{t2t4 cond}). We thus conclude from (\ref{t2t4 cond}) that
\begin{equation} \label{mixed ub} \frac{1}{2}  - 2 \delta_2 \leq t_1 \leq \frac{1}{2} + 2 \delta_2 \text{ and } \frac{1}{2} - \delta_2  \leq  t_2 + t_4 \leq \frac{1}{2} + 2 \delta_2.\end{equation} 
But then (\ref{4s1}) implies that 
\[t_4 \geq \frac{3}{4}  - t_1 - \delta_2 \geq \frac{1}{4} - 3 \delta_2,\]
thus (\ref{mixed ub}) gives
\begin{equation} \label{t2 ub} t_2 \leq  \frac{1}{2} + 2 \delta_2  - t_4 \leq \frac{1}{4} + 5 \delta_2.\end{equation} 
Feeding these estimates back into (\ref{ti sum}) gives
\begin{align*} t_3 & = 1 - t_1 - t_2 - t_4 \\
& \leq 1  - \frac{1}{2} + 2 \delta_2    - \frac{1}{2} + \delta_2 \\
& = 3 \delta_2. \end{align*} 
Next (\ref{quad fail}) implies that 
\begin{align*} t_2 & \geq \frac{3}{4}  - t_1 - t_3 \\
& \geq \frac{3}{4}  - \frac{1}{2} - 2 \delta_2 - 3 \delta_2 \\
&  = \frac{1}{4}   - 5 \delta_2
\end{align*} 
and feeding this back into (\ref{mixed ub}) gives 
\begin{align*} t_4 & \leq \frac{1}{2} + 2 \delta_2  - t_2 \\
& \leq \frac{1}{2}   - \frac{1}{4} + 7 \delta_2 \\
& = \frac{1}{4} + 7 \delta_2.
\end{align*} 
Now observe that
\begin{align*} \frac{t_1}{2} + t_3 & \leq \frac{1}{4}  + 4 \delta_2
\end{align*}
and
\begin{align*} \frac{t_2}{2} + t_4 & \geq \frac{1}{8} - \frac{5}{2} \delta_2  + \frac{1}{4}  - 3 \delta_2 = \frac{3}{8} - \frac{11}{2} \delta_2.
\end{align*}
Hence we have 
\begin{equation} \label{key quad ineq} \frac{t_1}{2} + t_3 \leq \frac{t_2}{2}  + t_4,
\end{equation}
provided again that we choose $\delta_2$ sufficiently small. Further, we have
\begin{align} \label{hb bd} \frac{t_2}{2} + t_4 & \leq \frac{1}{8} + \frac{5 \delta_2}{2}   + \frac{1}{4}  + 7 \delta_2   \leq \frac{t_1}{2} + t_3 + \frac{1}{8} + \frac{21 \delta_2}{2}  .
\end{align} 
We summarize this as:
\begin{proposition} Suppose that (\ref{quad fail}) and (\ref{4s1}) hold. Then 
\[\frac{t_1}{2} + t_3 \leq \frac{t_2}{2} + t_4 \leq \frac{1}{8} + \frac{21 \delta_2}{2} + \frac{t_1}{2} + t_3.\]
\end{proposition}

We now seek to apply Corollary 2 in \cite{B-HB} to (\ref{quad form}) or Proposition with the variables satisfying (\ref{dya range}). Indeed, we will have 
\[|v_1| \ll T_3, |w_1| \ll T_4, |u| \ll P^{-\frac{1}{2}} \max\left\{T_0^{\frac{1}{2}} T_4, T_1^{\frac{1}{2}} T_3\right\}, |\Delta| \asymp T_1 T_2 P, \Delta_0 \leq 4.\]
The last inequality comes from the observation that $P, w_0, v_0$ are pairwise co-prime. Proposition \ref{BHBcor} then gives the bound
\begin{equation} \label{quad form bd} \left(\frac{T_3 T_4 \max\{T_1^{1/2} T_3, T_2^{1/2} T_4\}}{P^{3/2} T_1 T_2} \right)^{1/3} \tau(v_0 v_1)\end{equation}
Since (\ref{hb bd}) implies
\[T_1 T_3^2 < T_2 T_4^2 \leq Z^{\frac{1}{4} + 21 \delta_2 } T_1 T_3^2,\]
equation (\ref{quad form bd}) and (\ref{ti sum}) give the bound 
\[\left(\frac{(Z/T_1 T_2) T_1^{1/4} T_3^{1/2} T_2^{1/4} T_4^{1/2} Z^{\frac{1}{16} + \frac{21 \delta_2}{4} }  }{P^{3/2} T_1 T_2} \right)^{1/3} Z^\ep \]
for any $\ep > 0$. Using 
\[T_3 T_4 \ll \frac{Z}{T_1 T_2}\]
we obtain 
\[\left(\frac{Z^{\frac{3 }{2} + \frac{1}{16} + \frac{21 \delta_2}{4} }}{(T_1 T_2)^{9/4}} \right)^{1/3} Z^\ep = \frac{Z^{\frac{25}{48} + \frac{7 \delta_2}{4}  }}{(T_1 T_2)^{3/4}} Z^\ep. \]
Multiplying by $T_1, T_2$ we obtain 
\[Z^{\frac{25}{48}  + \frac{7\delta_2}{4} } (T_1 T_2)^{1/4} Z^\ep. \]
By (\ref{mixed ub}) and (\ref{t2 ub}) we have
\[(T_1 T_2)^{1/4} \ll Z^{\frac{1}{4} \left(\frac{1}{2} + 2 \delta_2  + \frac{1}{4} + 2 \delta_2   \right)} = Z^{\frac{3}{16} + \delta_2  } \]
Therefore the total contribution is at most 
\begin{equation} \label{t bd} O_\ep \left( P^{-1/2} Z^{\frac{17}{24}  + \ep} \right) = O_\ep \left(\frac{X^{17/24 + \ep}}{P^{1/2 + 17/12}} \right) = O_\ep \left(\frac{X^{17/24 + \ep}}{P^{23/12}} \right),\end{equation}
and summing $P \ll X^{1/2}$ gives a total contribution of 
\begin{equation} \label{t bd2} O_\ep \left(X^{\frac{17}{24} + \ep} \right).
\end{equation}
by choosing $\delta_2$ sufficiently small with respect to $\ep$. This completes the proof of Theorem \ref{uniform 1}. \\

We remark that the estimate (\ref{t bd}) gives us more room than we need, and this can be used to good effect to control the number of curves with large Szpiro ratio.

\section{Curves with large Szpiro constant} 
\label{large szpiro}

Recall that for an elliptic curve $E$, the \emph{Spizro ratio} is defined to be
\begin{equation} \label{Spi cons} \beta_E = \frac{\log |\Delta(E)|}{\log C(E)}.
\end{equation}
Our goal in this section is to count curves in $\E_2$ having bounded conductor and such that the average of the Szpiro constants of $E$ and $\phi(E)$ is as large as possible. For a given $E = E_{a,b} \in \E_2$, put $\P_{\KTI}(E), \P_{\KT}(E), \P_{\KTS}(E)$ for the rational primes for which $E$ has Kodaira symbols $\KTI, \KT, \KTS$ respectively. Recall that
\[\C(E_{a,b}) = b(a^2 - 4b).\]
 Next we put $\Sigma(E)$ for the set of prime powers $p_i^{k_i}$ for $1 \leq i \leq m$ and $q_j^{\ell_j}, 1 \leq j \leq n$ such that 
\[P_1 Q_1 = \prod_{i=1}^m p_i^{k_i} \quad \text{exactly divides} \quad b\]
and
\[P_2 Q_2 = \prod_{j=1}^n q_j^{\ell_j} \quad \text{exactly divides} \quad c = a^2 - 4b.\]
Here
\[Q_1 = \prod_{j=1}^m p_j^{ k_j - 1} \text{ and } Q_2 = \prod_{j=1}^n q_j^{\ell_j - 1 }.\]
and
\[P_1 = p_1 \cdots p_m \text{ and } P_2 = q_1 \cdots q_n.\]

The condition of $C(E_{a,b}) \leq X$ then implies that
\begin{equation} \label{cond poly 2} |\C(E_{a,b})| = |b(a^2 - 4b)| \leq P_{\KTI}^2 P_{\KTS}^4 X Q_1 Q_2.
\end{equation}
If we write 
\[C^\prime(E_{a,b}) = \frac{C(E_{a,b})}{P_{\KTI}^2 P_{\KT}^2 P_{\KTS}^2} \text{ and } \C^\prime(E_{a,b}) = \frac{\C(E_{a,b})}{P_{\KTI}^4 P_{\KT}^2 P_{\KTS}^6} \]
then we can recast (\ref{cond poly 2}) as 
\begin{equation} \label{cond poly 3} |\C^\prime(E_{a,b})| \leq \frac{X Q_1 Q_2}{P_{\KTI}^2 P_{\KT}^2 P_{\KTS}^2}.
\end{equation}
Further, we can replace the variables $b,c$ with $P_{\KTI}^2 P_{\KT} P_{\KTS}^3 b, P_{\KTI}^2 P_{\KT} P_{\KTS}^3 c$ which gives the condition 
\[|b (P_{\KTI} P_{\KTS} a^2 - 4b)| \leq \frac{X Q_1 Q_2}{P_{\KTI}^2 P_{\KT}^2 P_{\KTS}^2}.\]
Put 
\[Y = \frac{X}{P_{\KTI}^2 P_{\KT}^2 P_{\KTS}^2}.\]

We may use the fact that elements in $\E_2$ are naturally connected by a rational 2-isogeny mapping 
\begin{equation} \label{isogeny} E_{a,b} \mapsto E_{-2a, a^2 - 4b},\end{equation} 
which preserves the conductor and essentially swaps the roles of $b, c$. \\ \\
Now put
\[b = P_1 Q_1 u \text{ and } c = P_2 Q_2 v,\]
where $u,v$ are square-free integers co-prime to $P_1, P_2$ respectively. It follows that our bound condition becomes
\begin{equation} |uv| \leq \frac{Q_1 Q_2 Y}{P_1 Q_1 P_2 Q_2} = \frac{Y}{ P_1 P_2}.
\end{equation}
Using the symmetry between $u,v$, we may assume that $|u| \leq |v|$ and therefore $|u| \leq \sqrt{Y/P_1 P_2}$. 

\subsection{Application of linear programming} 

Put
\begin{equation} \label{ai} \alpha_i = \frac{\log P_i}{\log X} \text{ for } i = 1,2,\end{equation}
\begin{equation} \label{bi} \beta_i = \frac{\log Q_i}{\log X} \text{ for } i = 1,2,\end{equation} 
and 
\[\upsilon = \frac{\log |u|}{\log X} \text{ and } \nu = \frac{\log |v|}{\log X}.\]
Further put 
\[\gamma_{\KTI} = \frac{\log P_{\KTI}}{\log X}, \gamma_{\KT} = \frac{\log P_{\KT}}{\log X} \text{, and } \gamma_{\KTS} = \frac{\log P_{\KTI}}{\log X}\]
Then the average of the Szpiro ratios of the pair $E_{a,b}, E_{-2a, a^2 - 4b}$ is 
\begin{align} \label{avgBE} \frac{\beta_E + \beta_{\phi(E)}}{2} & = \frac{\log |\Delta(E_{a,b})| + \log |\Delta(E_{-2a,a^2-4b})|}{2 \log C(E)} \\
& = \frac{6 (2 \gamma_{\KTI} + \gamma_{\KT} + 3 \gamma_{\KTS}) + 3 (\alpha_1 +  \beta_1 +  \upsilon + \alpha_2 + \beta_2 + \nu)}{2(2 (\gamma_{\KTI} + \gamma_{\KT} + \gamma_{\KTS}) + \alpha_1 + \upsilon + \alpha_2 + \nu)}. \notag \\
& = \frac{3}{2} + \frac{6 \gamma_{\KTI}  + 12 \gamma_{\KTS}  +  3\beta_1 + 3\beta_2}{2(2 (\gamma_{\KTI} + \gamma_{\KT} + \gamma_{\KTS}) + \alpha_1 + \upsilon + \alpha_2 + \nu)}. \notag
\end{align} 
At this point, we outline our strategy to use linear programming to prove Theorem \ref{MT2}. We introduce the sets
\begin{equation} \label{a1a2half} \E_2(\delta; X) = \left \{E_{a,b} \in \E_2(X), \delta \leq \alpha_1 + \alpha_2 \leq \frac{1}{2} - \delta \right\}
\end{equation}
and
\begin{equation} \label{aibi102} \E_2(\delta, r; X) = \left\{E_{a,b} \in \E_2(\delta; X), \beta_1 < \alpha_1 + r, \beta_2 < \alpha_2 + r\right \}.
\end{equation}
Let $N(\delta; X) = \#\E_2(\delta; X)$ and $N(\delta, r; X) = \# \E_2(\delta, r; X)$. We shall prove: 

\begin{proposition} \label{halflem} We have 
\[N(\delta; X) = \# \E_2(\delta; X) = O_\delta \left(X^{\frac{3 - \delta}{4}} \right).\]
\end{proposition}

An application of linear programming bounds yields:

\begin{proposition} \label{LPmain} Suppose $E = E_{a,b} \not \in \E_2(\delta; r, X)$ and $\alpha_1 + \alpha_2 > \delta$. Then 
\[\frac{\beta_{E} + \beta_{\phi(E)}}{2} > \frac{3}{2} - 3 \delta + 3r.\]
\end{proposition}

Theorem \ref{szpiro uniform} then follows from Proposition \ref{LPmain} and the following proposition: 

\begin{proposition} \label{nearquadprop} We have 
\[N\left(\delta, \frac{1}{102}; X \right) = o_\delta \left(X^{\frac{3}{4}} \right).\]
\end{proposition}

Let us describe how to prove Theorem \ref{szpiro uniform}. Let $\E_{2,\kappa}(\theta; X)$ be the set given by (\ref{szpiroct}), namely
\[\E_{2,\kappa}(\theta; X) = \left \{E_{a,b} \in \E_{a,\kappa} : \frac{3}{2} + \theta < \frac{\beta_E + \beta_{\phi(E)}}{2} \leq \kappa \right\}. \]
 Proposition \ref{LPmain} then asserts that 
\begin{equation} \E_{2,\kappa}(\theta; X) \subseteq \left \{E_{a,b} \in \E_2(X) :  \alpha_1 + \alpha_2 \leq \delta  \right\} \cup \E_2(\delta, r; X).
\end{equation}
Thus the proof of Theorem \ref{szpiro uniform} requires us to estimate the size of the two sets on the right. The set on the left is easy to handle, as we will see in the following lemma: 

\begin{lemma} We have 
\[\# \{E_{a,b} \in \E_2(X) : \alpha_1 + \alpha_2 \leq \delta\} = O_\delta \left(X^\delta (\log X)^3 \right).\]
\end{lemma} 

\begin{proof} We further partition the set in the lemma into two ranges, namely 
\[\S_1 = \{E_{a,b} \in \E_2(X) : \alpha_1 + \alpha_2 \leq (\log \log X)^{-1} \} \quad \text{and} \quad \S_2 = \{E_{a,b} \in \E_2(X) : (\log  \log X)^{-1} \leq \alpha_1 + \alpha_2 \leq \delta\}.\]
We estimate the size $\S_1$ as follows. Using the fact that our Szpiro ratio is bounded by $2\kappa$, we find that given $P_1, P_2 \leq X^{1/\log \log X}$ there are $O((\log \log X)^2)$ choices for $Q_1, Q_2$. There are thus $O_\ep\left(X^\ep \right)$ possible curves in $\S_1$. \\ \\
To estimate the size of $\S_2$, we cut into dyadic ranges $\left[X^\psi/2,  X^\psi\right]$ with 
\[(\log \log X)^{-1} \ll \psi \leq \delta.\]
In this range, there are $O(X^{\psi} \log X)$ choices to choose $P_1, P_2$ (namely, we select a square-free integer $m$ satisfying $X^{\psi} \ll m \ll X^\psi$, and then take divisors). Again, our bound on the Szpiro ratio implies there are at most $\left(2\lceil  \kappa/\psi \rfloor + 1\right)^4$ choices for $Q_1, Q_2$. Thus, this gives a total contribution of $O\left(X^\psi (\log X)^2 \right)$ contributions. Summing over dyadic ranges, we find that each dyadic interval contributes $O\left(X^\delta (\log X)^3\right)$ possible curves. We conclude that 
\[\# \S_1 + \# \S_2 = O_\delta \left(X^\delta (\log X)^3 \right).\]
\end{proof} 

Thus, for $\delta$ small, this is comfortably an error term. \\ 

It remains to prove Propositions \ref{halflem} and \ref{LPmain}. We will prove the latter first. 

\subsection{Proof of Proposition \ref{LPmain}} We will assume that Lemma \ref{halflem} holds. For each $E \in \E_2(\delta, r; X)$ we have
\begin{equation} 2 (\gamma_{\KTI} + \gamma_{\KT} + \gamma_{\KTS}) + \alpha_1 + \upsilon + \alpha_2 + \nu \leq 1.
\end{equation}
Using symmetry, we may assume that 
\begin{equation} \label{P1Q1P2Q2} \alpha_1 + \beta_1 \geq \alpha_2 + \beta_2.
\end{equation}

By Lemma \ref{halflem} we obtain the constraint 
\begin{equation} \label{a1a2} \alpha_1 + \alpha_2 \geq \frac{1}{2} - \delta.
\end{equation}
Thus each $E \in \E_2(\delta, r; X)$ satisfies the constraints of the following linear program

\begin{equation} \begin{matrix} \min & 6 \gamma_{\KTI}  + 12 \gamma_{\KTS}  +  3\beta_1 + 3\beta_2 \\ 
\text{subject to } & -2 (\gamma_{\KTI} + \gamma_{\KT} + \gamma_{\KTS}) - \alpha_1 - \upsilon - \alpha_2 - \nu  \geq - 1 \\
& \alpha_1 + \alpha_2  \geq \frac{1}{2} - \delta \\
& \alpha_1 + \beta_1  \geq \alpha_2 + \beta_2 \\
& \beta_2 \geq \alpha_2 + r \text{ and } \beta_1 \geq \alpha_1 + r.
\end{matrix} 
\end{equation}

We write this in standard form 
\begin{equation} \label{primalLP} \begin{matrix} \min & c^T x \\
\text{subject to} & Ax \geq b \\
& x \geq \mathbf{0}.
\end{matrix}  
\end{equation}
Here the vector $x$ is given by
\[(\gamma_{\KTI}, \gamma_{\KT}, \gamma_{\KTS}, \alpha_1, \alpha_2, \beta_1, \beta_2, \upsilon, \nu) = (x_1, x_2, x_3, x_4, x_5, x_6, x_7, x_8, x_9),\]
the objective vector is given by 
\[c^T = (6,0,12, 0, 0, 3, 3, 0, 0),\]
the coefficient matrix given by 
\[A = \begin{bmatrix} -2 & -2 & -2 & -1 & -1 & 0 & 0 & -1 & -1 \\ 0 & 0 & 0 & 1 & 1 & 0 & 0 & 0 & 0 \\ 0 & 0 & 0 & 1 & -1 & 1 & -1 & 0 & 0 \\ 0 & 0 & 0 & -1 & 0 & 1 & 0 & 0 & 0 \\ 0 & 0 & 0 & 0 & -1 & 0 & 1 & 0 & 0 \end{bmatrix}\]
and the constraint vector $b$ given by $b = (-1, 1/2 - \delta,0, r, r)^T$. The dual program is then given by 
\begin{equation} \label{dualLP} \begin{matrix} \max & y_1 + \left(\frac{1}{2} - \delta \right) y_2 + r y_4 + r y_5 \\ 
\text{subject to } & A^T y \leq c \\
& y_1, y_2, y_3, y_4, y_5 \geq 0.
\end{matrix} 
\end{equation}
Taking 
\[x^\ast = \frac{1}{2} \left( 0,0,0, \frac{1}{2} - \delta, \frac{1}{2} - \delta, \frac{1}{2} - \delta + r, \frac{1}{2} - \delta + r, \frac{1}{2} + \delta, \frac{1}{2} + \delta \right)^T\]
we see that all constraints are satisfied, whence $x^\ast$ is feasible. We obtain the primal objective value of 
\[c^T x^\ast = \frac{3}{2} - 3 \delta + 3r.\]
Taking 
\[y^\ast = (0,3,0,3,3)^T\]
we see that $y^\ast$ is dual feasible, with objective value $b^T y^\ast = 3/2 - 3 \delta + 3r$. Thus, by weak duality, $3/2 - 3 \delta + 3r$ is the optimal value of the primal program (\ref{primalLP}). We conclude that for any $E \not \in \E_2(\delta, r; X)$ that $(\beta_{E} + \beta_{\phi(E)})/2 > 3/2 - 3 \delta + 3r$. This completes the proof of Proposition \ref{LPmain}. \\

Note that 
\[\E_2(\delta; X) = \bigcap_{r > 0} \E_2(\delta, r; X).\]
Thus, Lemma \ref{halflem} immediately implies that Theorem \ref{szpiro uniform} holds with 
\[\kappa < \frac{3}{2} + \frac{3}{4} - 3 \delta = \frac{9}{4} - 3 \delta\]
for all $\delta$ sufficiently small. \\

We will now prove Proposition \ref{halflem}. 

\subsection{Proof of Proposition \ref{halflem}} 

In this subsection we follow the strategy of \cite{SSW}, and fix sets of primes $\P = \P_{\KTI} \cup \P_{\KT} \cup \P_{\KTS}$ corresponding to the Kodaira symbols in the subscripts and primes $\Sigma = \P_1 \cup \P_2$ corresponding to primes having multiplicative bad reduction. We are then interested in counting the number of elements in 
\begin{equation} \label{EPS} \E_2(\P, \Sigma)(X) = \{E_{a,b}: C(E_{a,b}) \leq X, \gcd(a,b) = 1, P_1 Q_1 \vert \vert b, P_2 Q_2  \vert \vert a^2 - 4b,\end{equation}
\[ \C(E_{a,b})/(P_{\KTI}^2 P_{\KT}^2 P_{\KTS}^2 P_1 P_2 Q_1 Q_2) \text{ is square-free}\}. \]
Let us put 
\[N(\P, \Sigma)(X) = \# \E_2(\P, \Sigma)(X).\]

We therefore obtain the decomposition of $N(\P, \Sigma)(X)$ as a sum
\begin{equation} \label{NPS} N(\P,\Sigma)(X) = \sum_{|u| \leq \sqrt{Y/P_1 P_2}} \N_u \left(\frac{Q_2 Y}{P_1 |u|} \right),
\end{equation}
where $\N_u(Z)$ is the number of solutions to the inequality 
\[|P_{\KTI} P_{\KTS} a^2 - 4 P_1 Q_1 u| \leq Z \text{ subject to } P_{\KTI} P_{\KTS} a^2 - 4P_1 Q_1 u \equiv 0 \pmod{P_2 Q_2}.\]
We will establish the following estimate for $N(\P, \Sigma)(X)$:

\begin{proposition} \label{LSprop} Let $N(\P, \Sigma)(X)$ be given as in (\ref{NPS}). We then have the bounds:

\begin{equation} \label{u bd} \sum_{|u| \leq \sqrt{Y/P_1 P_2}} \N_u \left(\frac{Q_2 Y}{P_1 |u|} \right) \ll \frac{Y^{1/2}}{(P_1 P_2)^{1/2}} +\end{equation}
\[  \begin{cases} \dfrac{Y}{(P_1 P_2)(P_2 Q_2)} & \text{if } P_1 Q_1 \ll P_2 Q_2 \text{ and } Y \ll (P_1 P_2)(P_2 Q_2)^2 \\ \\ 
\dfrac{Y^{3/4}}{(P_1 P_2)^{3/4} (P_2 Q_2)^{1/2}} & \text{if } P_1 Q_1 \ll P_2 Q_2 \text{ and }  Y \gg (P_1 P_2)(P_2 Q_2)^2 \\ \\ 
\dfrac{Y}{(P_1 P_2)(P_2 Q_2)} + \dfrac{Y^{3/4}}{P_1 P_2 (Q_1 Q_2)^{1/4}} & \text{if } P_2 Q_2 \ll P_1 Q_1 \text{ and } Y \ll Q_2 Q_1^{-1} (P_2 Q_2)^2 
\\ \\ \dfrac{Y^{3/4}}{(P_1 P_2)(P_2 Q_2)^{1/4}} + \dfrac{Y^{3/4}}{(P_1 P_2)(Q_1 Q_2)^{1/4}} & \text{if } P_2 Q_2 \ll P_1 Q_1 \text{ and } Y \gg Q_2 Q_1^{-1} (P_2 Q_2)^2.  \end{cases} \]
and
\begin{equation} \label{v bd}  \sum_{|v| \leq \sqrt{Y/P_1 P_2}} \N_v \left(\frac{Q_2 Y}{P_1 |v|} \right) \ll \frac{Y^{1/2}}{(P_1 P_2)^{1/2}} +\end{equation}
\[  \begin{cases} \dfrac{Y}{(P_1 P_2)(P_1 Q_1)} & \text{if } P_2 Q_2 \ll P_1 Q_1 \text{ and } Y \ll (P_1 P_2)(P_1 Q_1)^2 \\ \\ 
\dfrac{Y^{3/4}}{(P_1 P_2)^{3/4} (P_1 Q_1)^{1/2}} & \text{if } P_2 Q_2 \ll P_1 Q_1 \text{ and }  Y \gg (P_1 P_2)(P_1 Q_1)^2 \\ \\ 
\dfrac{Y}{(P_1 P_2)(P_1Q_1)} + \dfrac{Y^{3/4}}{P_1 P_2 (Q_1 Q_2)^{1/4}} & \text{if } P_1 Q_1 \ll P_2 Q_2 \text{ and } Y \ll Q_1 Q_2^{-1} (P_1 Q_1)^2 
\\ \\ \dfrac{Y^{3/4}}{(P_1 P_2)(P_1 Q_1)^{1/4}} + \dfrac{Y^{3/4}}{(P_1 P_2)(Q_1 Q_2)^{1/4}} & \text{if } P_1 Q_1 \ll P_2 Q_2 \text{ and } Y \gg Q_1 Q_2^{-1} (P_1 Q_1)^2.  \end{cases} \]
\end{proposition}

\begin{proof}

We estimate $\N_u(Q_2 Y P_1^{-1} |u|^{-1})$ in several different ways, depending on the relative sizes of the quantities involved. \\ \\
First note that 
\[P_1 Q_1 \ll P_2 Q_2\]
is equivalent to 
\[\frac{X^{1/2}}{(P_1 P_2)^{1/2}} \ll \frac{Q_2^{1/2} X^{1/2}}{P_1 Q_1^{1/2}}.\]
Further we have
\begin{equation} \label{little case} P_1 Q_1 |u| \ll \frac{Q_2 X}{P_1 |u|}. \end{equation} 
Thus $a$ is constrained to a union of intervals of total length $O (\sqrt{Q_2 Y/P_1 |u|})$. It follows that
\[\N_u \left(\frac{Q_2 Y}{P_1 |u|} \right) = O \left(\sqrt{\frac{Q_2 Y}{P_1 |u|}} \frac{1}{P_2 Q_2} + 1 \right).\]
Next we consider the possibility that 
\[\frac{\sqrt{Y}}{P_2 \sqrt{P_1 Q_2 |u|}} \ll 1.\]
This is equivalent to 
\[|u| \gg \frac{Y}{P_2^2 P_1 Q_2}.\]
This is only relevant if 
\[\frac{Y}{P_2^2 P_1 Q_2} \ll \sqrt{\frac{Y}{P_1 P_2}}.\]
This condition implies 
\[Y \ll P_2^3 P_1 Q_2^2 = (P_1 P_2)(P_2 Q_2)^2.\]
If this holds then we obtain the estimate 
\begin{align} \label{X ontop} \sum_{|u| \leq \sqrt{Y/P_1 P_2}} \N_u \left(\frac{Q_2 Y}{P_1 |u|} \right) & \ll \sqrt{\frac{Y}{P_1 P_2}} + \sum_{|u| \leq Y/(P_1 P_2)(P_2 Q_2)} \frac{Y^{1/2}}{P_2 (P_1 Q_2 |u|)^{1/2}}  \\
& = \frac{Y^{1/2}}{(P_1 P_2)^{1/2}} + \frac{Y}{(P_1 P_2)(P_2 Q_2)}. \notag
\end{align} 
If on the other hand 
\[(P_1 P_2)(P_2 Q_2)^2 \ll Y,\]
then we obtain 
\begin{align*} \sum_{|u| \leq \sqrt{Y/P_1 P_2}} \N_u \left(\frac{Q_2 Y}{P_1 |u|} \right) & \ll \sqrt{\frac{Y}{P_1 P_2}} + \sum_{|u| \leq \sqrt{Y/P_1 P_2}} \frac{Y^{1/2}}{P_2 (P_1 Q_2 |u|)^{1/2}}  \\
& = \frac{Y^{1/2}}{(P_1 P_2)^{1/2}} + \frac{Y^{3/4}}{(P_1 P_2)^{3/4}(P_2 Q_2)^{1/2}}.
\end{align*} 
If
\[P_1 Q_1 \gg P_2 Q_2\]
then
\[\frac{Q_2^{1/2} Y^{1/2}}{P_1 Q_1^{1/2}} \ll \frac{Y^{1/2}}{(P_1 P_2)^{1/2}}.\]
In the range 
\[1 \leq |u| \ll \frac{Q_2^{1/2} Y^{1/2}}{P_1 Q_1^{1/2}} \]
one has to consider the possibility, as above, that 
\[\frac{Y}{(P_1 P_2)(P_2 Q_2)} \ll \frac{Q_2^{1/2} Y^{1/2}}{P_1 Q_1^{1/2}}.\]
This gives 
\begin{align} \label{X ontop2} \sum_{|u| \leq \sqrt{Y/P_1 P_2}} \N_u \left(\frac{Q_2 Y}{P_1 |u|} \right) & = \sum_{|u| \leq Y/(P_1 P_2)(P_2 Q_2)} \frac{Y^{1/2}}{P_2 (P_1 Q_2 |u|)^{1/2}} + O \left( \sum_{Y/(P_1 P_2)(P_2 Q_2) < |u| \leq \sqrt{Y/P_1 P_2}} 1 \right) \\
& \ll \sqrt{\frac{Y}{P_1 P_2}} + \sum_{|u| \leq Y/(P_1 P_2)(P_2 Q_2)} \frac{Y^{1/2}}{P_2 (P_1 Q_2 |u|)^{1/2}}  \notag \\
& = \frac{Y^{1/2}}{(P_1 P_2)^{1/2}} + \frac{Y}{(P_1 P_2)(P_2 Q_2)}. \notag
\end{align} 
which is the same bound as (\ref{X ontop}). Otherwise 
\begin{align*} \sum_{|u| \leq Q_2^{1/2} Y^{1/2}/P_1 Q_1^{1/2}} \N_u \left(\frac{Q_2 Y}{P_1 |u|} \right) & \ll \sqrt{\frac{Y}{P_1 P_2}} + \sum_{|u| \leq Q_2^{1/2} Y^{1/2}/P_1 Q_1^{1/2}} \frac{Y^{1/2}}{P_2 (P_1 Q_2 |u|)^{1/2}}  \\
& = \frac{Y^{1/2}}{(P_1 P_2)^{1/2}} + \frac{Y^{3/4}}{(P_1 P_2)(P_2 Q_2)^{1/4}}.
\end{align*} 
In the range
\begin{equation} \label{big case} \frac{Q_2^{1/2} Y^{1/2}}{P_1 Q_1^{1/2}} \ll |u| \leq \frac{Y^{1/2}}{(P_1 P_2)^{1/2}} \end{equation} 
we have that $a^2$ is constrained by 
\[4P_1 Q_1 |u| - \frac{Q_2 Y}{P_1 |u|} \leq a^2 \leq 4 P_1 Q_1 |u| + \frac{Q_2 Y}{P_1 |u|}, \]
so $a$ is constrained in two intervals of length 
\[O \left(\frac{Q_2 Y/(P_1 |u|)}{\sqrt{P_1 Q_1 |u|}}  \right) = O \left(\frac{Q_2 Y}{(P_1 |u|)^{3/2} Q_1^{1/2}} \right).\]
Dividing by $P_2 Q_2$ to account for the congruence, this means that for a fixed $u$ satisfying (\ref{big case}), we have
\[\N_u \left(\frac{Q_2 X}{P_1 |u|} \right) = O \left(\frac{X}{Q_1^{1/2} P_2 (P_1 |u|)^{3/2}} + 1 \right).\]
Summing over the range (\ref{big case}) then gives the bound 
\[O \left(\frac{X^{3/4}}{(Q_1 Q_2)^{1/4} P_1 P_2} + \frac{X^{1/2}}{(P_1 P_2)^{1/2}} \right).\]
\end{proof}

Note that we can determine $Q_1, Q_2$ given $P_1, P_2$ in $O_\ep \left(X^\ep \right)$ ways. We may then restrict $P_1, P_2$ into intervals $[T_1, 2T_1), [T_2, 2T_2)$ respectively. Let $\E_2(\P,\Sigma; T_1, T_2)(X)$ denote the set of curves with $P_1, P_2$ in that range. \\ \\
Now suppose that 
\begin{equation} \label{T1T2 bd} Y^\delta \ll_\delta T_1 T_2 \ll_\delta Y^{1/2 - \delta}
\end{equation}
for some $\delta > 0$. The condition
\[(P_2 Q_2)^2 (P_1 P_2) \gg Y \text{ implies that } P_2 Q_2 \gg \frac{Y^{1/2}}{(P_1 P_2)^{1/2}},\]
hence
\[\frac{Y}{(P_1 P_2)(P_2 Q_2)} \ll \frac{Y^{1/2}}{(P_1 P_2)^{1/2}}.\]
Further, $P_2 Q_2 \geq P_1 Q_1$ and $Q_i \geq P_i$ for $i = 1,2$ imply that 
\[P_2 Q_2 \geq \sqrt{P_1 Q_1 P_2 Q_2} = \sqrt{Q_1 Q_2} \cdot \sqrt{P_1 P_2} \geq P_1 P_2.\]
It follows that
\[\frac{Y^{3/4}}{(P_1 P_2)^{3/4} (P_2 Q_2)^{1/2}} \leq \frac{Y^{3/4}}{(P_1 P_2)^{5/4}}.\]
Thus, the first two lines of (\ref{u bd}) can be replaced with
\[\frac{Y^{3/4}}{(P_1 P_2)^{5/4}} + \frac{Y^{1/2}}{(P_1 P_2)^{1/2}}.\]
If 
\[P_1 Q_1 \ll_\delta Y^{1/3 - \delta}\]
then certainly $Q_1 \ll_\delta Y^{1/3 - \delta}$, and therefore
\[Q_1 Q_2^{-1} (P_1 Q_1)^2 \ll_\delta Y^{1 - 3\delta}.\]
Thus $(P_1 Q_1)^2 Q_1 Q_2^{-1} \gg Y$ implies that 
\[P_1 Q_1 \gg_\delta Y^{1/3 - \delta}\]
for any $\delta > 0$. The third line of (\ref{v bd}) can then be replaced with 
\[\frac{Y^{2/3 + \delta}}{P_1 P_2} + \frac{Y^{3/4}}{(P_1 P_2)^{5/4}}.\]
Finally, we obtain the bound 
\begin{equation} |\E_2(\P,\Sigma; T_1, T_2)(X) | \ll_\delta \frac{Y^{3/4}}{(T_1 T_2)^{1/4}} + (T_1 T_2 Y)^{1/2} + Y^{2/3 - \delta}.
\end{equation}
If (\ref{T1T2 bd}) holds then we obtain the bound 
\begin{equation} \label{E2PSbd1} |\E_2(\P,\Sigma; T_1, T_2)(X)| \ll_\delta X^{\frac{3 - \delta}{4}}\end{equation} 
since $T_1 T_2 \ll Y^{1/2 - \delta} \ll_\delta X^{1/2 - \delta}$. Summing over dyadic intervals gives Proposition \ref{halflem}.

\subsection{Proof of Proposition \ref{nearquadprop}}

To push beyond the $9/4$ barrier, we note that the optimal solution given by Proposition \ref{LPmain} with $r = 0$ comes from $\alpha_1, 
\alpha_2, \beta_1, \beta_2$ all nearly equal to $1/4$.  This means that the curves at the boundary all satisfy $P_i \asymp Q_i$, so most of the exponents in $Q_i$ are equal to one. To measure this, we impose the condition that 
\begin{equation} \label{near square} Q_i \leq P_i X^\nu \text{ for } i = 1,2.
\end{equation}
That is, the primes of multiplicative bad reduction with higher multiplicity mostly divide the discriminant at most twice.\\

Put $R_i$ for the product of all primes $p$ such that $p^2 \vert \vert b, c$ respectively, and put $S_i$ so that $P_i Q_i = R_i^2 S_i$ for $i = 1,2$. Then (\ref{near square}) implies that we have
\begin{equation} b = R_1^2 S_1 u, c = R_2^2 S_2 v.
\end{equation}
Note that 
\[S_i = \frac{P_i Q_i}{R_i^2} \leq \left(\frac{P_i}{R_i} \right)^2 X^\nu \]
by (\ref{near square}). Now observe that $P_i/R_i$ is bounded by $Q_i/P_i \leq X^\nu$, hence
\begin{equation} \label{Si bd} S_i \leq X^{3 \nu} \text{ for } i = 1,2. 
\end{equation}
We then obtain
\[c = a^2 - 4b \Rightarrow S_2 R_2^2 u = a^2 - 4 S_1 R_1^2 v\]
with $S_1, S_2 \leq X^{3 \nu}$. The height condition on the conductor implies 
\[|uv| \rad(S_1 S_2) R_1 R_2 \leq X. \]
Note that $\rad(S_1 S_2) \geq 1$, and since $S_1 S_2 \leq X^{6\nu}$, it follows that 
\[(uS_1) R_1 (vS_2) R_2 \leq X^{1 + 6 \nu}.\]
To finish the proof, it suffices to note that Theorem \ref{uniform 1} and its proof are given in terms of an auxiliary parameter $Z$. Replacing $Z$ with $X^{1 + 6\nu} P^{-2}$ in (\ref{t bd}) gives 
\begin{equation} \left(\frac{X^{1 + 6\nu}}{P^2} \right)^{\frac{17}{24}  + \ep} 
\end{equation} 
Choosing $\ep$ appropriately then gives Proposition \ref{nearquadprop}.

\section{$3$-Selmer elements of elliptic curves with a marked $2$-torsion point} 
\label{3-sel} 

We follow the parametrization obtained by Bhargava and Ho \cite{B-Ho} for $3$-Selmer elements of elliptic curves in the family $\E_2$ with a marked rational 2-torsion point. In particular, they proved that the $18$-dimensional space 
\[\bQ^3 \otimes \Sym_2(\bQ^3)\]
of triples of $3 \times 3$ symmetric matrices represent elements of the $3$-Selmer groups of elliptic curves with a marked 2-torsion point. Further they showed that $3$-Selmer elements admit integral representatives, so we may take triples of integral $3 \times 3$ symmetric matrices instead. Moreover we must take \emph{equivalence classes} of a $\GL_3(\bZ) \times \GL_3(\bZ)$-action obtained by letting $\GL_3$ act simultaneously on the three ternary quadratic forms defined by each of the matrices in the triple, and on the triple itself. In order to obtain a faithful action we must mod out by a certain element of order $3$, which gives us an action by the group $\SL_3^2(\bZ)/\mu_3$. \\

Using a product of two Siegel fundamental domains for the action of $\SL_3(\bZ)$ on $\SL_3(\bR)$ one then obtains a fundamental domain for the action of $\SL_3^2(\bZ)$ on $V_3$, the space of triples of $3 \times 3$ symmetric matrices. Further, they proved that the problematic regions with regards to geometry of numbers, namely the so-called cusps, contain only irrelevant points. Using the now-standard ``thickening and cutting off the cusp" method of Bhargava (see \cite{Bha1}, \cite{B-Ho}, \cite{BhaSha1}, etc.), one then obtains an expression of the shape 
\[N(S;X) = \frac{1}{C_{G_0}} \int_{g \in N^\prime(a) A^\prime} \# \{x \in S^{\text{irr}} \cap E(\nu, \alpha, X)\}dg.\] 
Here $N(S;X)$ counts the number of $G(\bZ)$-equivalence classes of irreducible elements in $S$ having height bounded by $X$, and 
\[E(\nu, \alpha, X) = \nu \alpha G_0 R \cap \{x \in V_3^{(i)} : H(x) < X\},\]
where $G_0$ is a compact, semi-algebraic left $K$-invariant subset of the group $G(\bR) = \GL_3^2(\bR)/\mu_3$ which is the closure of a non-empty, connected open set and such that every element has determinant at least one and $R = R^{(i)}$ denotes a connected and bounded set representing real-orbits of $G(\bR)$ acting on $V_3(\bR)$. Further, we have
\begin{align} \label{Iwasawa} 
K & = \text{subgroup } \SO_3(\bR) \subset \GL_3^+(\bR) \text{ of orthogonal transformations}; \\
A^\prime & = \{\alpha(t, u) : t, u > c\}, \text{where} \notag \\
& \alpha(t,u) = \begin{pmatrix} t^{-2} u^{-1} & & \\ & tu^{-1} & \\ & & tu^2 \end{pmatrix}; \notag \\
N^\prime & = \{\nu(x, x^\prime, x^{\prime \prime}) : (x, x^\prime, x^{\prime \prime}) \in I^\prime(\alpha)\}; \notag \\
& \text{where } \nu(x, x^\prime, x^{\prime \prime}) = \begin{pmatrix} 1 & & \\ x & 1 & \\ x^\prime & x^{\prime \prime} & 1 \end{pmatrix}. \notag
\end{align}
Here $I^\prime(\alpha)$ is a measurable subset of $[-1/2, 1/2]^3$ dependent only on $\alpha \in A^\prime$ and $c > 0$ is an absolute constant. Now the set $\nu \alpha G_0 R$ is then seen as the image of $\nu \alpha G_0$ acting on $R$, where $\nu \alpha G_0$ is the left-translation of $G_0$ by $\nu \in N^\prime$ and $\alpha \in A^\prime$. \\

A key observation, implicit in previous works on this subject but seemingly not written down explicitly, is that obtaining an asymptotic formula for $N(S; X)$ boils down to an acceptable estimate for the number of integral points inside $E(\nu, \alpha, X)$. Note that the set $\nu \alpha G_0 R$ is independent of the choice of height $H$. The height function used by Bhargava and Ho in \cite{B-Ho} has degree $36$: it is the defined to be the maximum of $|a_2|^{6}, |a_4|^3$ where $a_2, a_4$ are degree $6$ and degree $12$ homogeneous polynomials in the entries of $B \in V_3(\bR)$ respectively. For us the height is given by the conductor polynomial, which we recall is given by (\ref{condpoly}). For $a = a_2, b = a_4$ the conductor polynomial evidently has degree $24$ instead of $36$. \\

The key observations are the following: we can use the structure of the action of $G(\bR)$ on $V$ to transform the problem of counting in a bounded region in $V$ to a corresponding problem of counting in a region in $\bR^2$, where our earlier arguments of counting elliptic curves in the related region apply. As long as the counting problem for Selmer elements is compatible with the error estimates obtained in the essentially purely geometric reduction argument, we are able to obtain a suitable count of the total number of Selmer elements and thus obtain our desired outcome. \\

We define the set $\R_X$ analogously to the way it is defined in \cite{B-Ho}, namely
\begin{equation} \R_X(h) = \F h R \cap \{B \in V_3(\bR) : |\C(B)| < X\}, \R_X = \R_X(1).
\end{equation} 

We remark that the condition $|\C(B)| < X$ does \emph{not} define a bounded set in terms of $a_2, a_4$ as is the case with Bhargava and Ho's choice of height function in \cite{B-Ho}. However, since $\C(B)$ is defined in terms of the invariant functions $a_2, a_4$, that the value of $\C(B)$ homogeneously expands with $B$. In particular, for any $g \in G(\bR)$ the value of $\C(g \circ B)$ depends only on $\deg g$ and $\C(B)$, and not on the $u,t,x$ parameters in (\ref{Iwasawa}). \\

The key lemma we require, which is proved in its essential form as Proposition 7.3 in \cite{B-Ho}, is the following:

\begin{proposition} \label{3selm} Let $h$ take a random value in $G_0$ uniformly with respect to the Haar measure $dg$. Then the expected number of elements $B \in \F h R \cap V(\bZ)$ such that $|\C(B)| <X$, $|a_4|, |a_2^2 - 4a_4| \geq 1$, and $b_{111} \ne 0$ is equal to $\Vol(\R_X) + O \left(X^{\frac{17}{24}} \right)$. 
\end{proposition}

\begin{proof}
In order to obtain this Proposition we essentially follow along the same lines as in the proof of Propositions 7.3 and 7.6 in \cite{B-Ho}. However, some remarks are in order. First, the homogeneous expansion condition still holds, regardless of the choice of height function. In particular $|\C(B)| \leq X$ and the assumption that $|b_{111}| \geq 1$ imply (25) in \cite{B-Ho}, namely the existence of an absolute constant $J$ such that for each entry $b_{ijk}$ of $B$ we have
\[|b_{ijk}| \leq J w(b_{ijk}) X^{\frac{1}{24}}. \]
This condition implies that the volume computations and projection arguments in \cite{B-Ho} apply equally well in our setting. In particular, the proof essentially follows along the same lines as given in \cite{B-Ho}. The key difference is that in the proofs of Proposition 7.3-7.6 in \cite{B-Ho} one must replace the parameter $k = 36$ with $k = 24$; all other details are the same. 
\end{proof}

An additional issue is that we are looking to exclude the points satisfying $a_4 = 0$ or $a_2^2 - 4a_4 = 0$. We have the following lemma to address this issue: 

\begin{lemma} For $B \in \bZ^3 \otimes \Sym_2(\bZ^3)$, let $a_2, a_4$ be the invariant polynomials given in \cite{B-Ho}. Then the number of elements $B \in \F h R \cap V(\bZ)$ with $|\C(B)| < X$ and $a_4(a_2^2 - 4a_4) = 0$ is $O\left(X^{\frac{17}{24}} \right)$.
\end{lemma} 

\begin{proof} To solve the equation $a_4 = 0$ or $a_2^2 - 4a_4 = 0$ we fix all but one coefficient. As in \cite{B-Ho} this amounts to $O \left(X^{\frac{17}{24}} \prod_{b}^\prime w(b) \right)$ choices, where the prime indicates that precisely one term is missing in the product. By choosing to remove the largest weight we can guarantee that $\prod_b^\prime w(b) \leq 1$. Having chosen $17$ of the variables, the remaining variable (regardless of size) can be chosen in at most $24$ ways by solving the corresponding polynomial equation in a single variable. 
\end{proof} 

Next, we require a change-of-measure formula. This is proved in \cite{B-Ho} as well, namely it is their Proposition 7.7. We state this again for completeness. 

\begin{proposition}[Proposition 7.7, \cite{B-Ho}] \label{BHo77}  There exists a rational number $\mathfrak{J}$ such that, for any measurable function $\phi$ on $V(\bR)$, we have 
\begin{equation} \frac{|\fJ|}{n_i} \int_{R} \int_{G(\bR)} \phi \left(g \cdot V (\vec a) \right) dg d \vec a = \int_{V} \phi(v) dv.
\end{equation} 
\end{proposition} 

In particular, Proposition \ref{BHo77} implies 
\begin{align*} \int_{\R_X}  dv & = \int_{\F \cdot R(X)} dv \\
& = |\fJ| \cdot \int_{R(X)} \int_{\F} dg d \vec a \\
& = |\fJ| \cdot \Vol(G(\bZ) \setminus G(\bR)) \cdot \int_{R(X)} d\vec a.
\end{align*}

Combining these results we obtain 
\begin{equation} \label{3-Sel count}N \left(V(\bZ); X\right) = |\fJ| \cdot \Vol(G(\bZ) \setminus G(\bR)) \cdot A_\infty(X) + O \left(X^{\frac{17}{24}}\right). \end{equation} 

We emphasize that the quantities $|\fJ|, \Vol(G(\bZ)\setminus G(\bR))$ are universal constants and therefore do not depend on the choice of the height function $H$, nor on the parameter $X$. Finally we note that the \emph{upper bound sieve} is essentially trivial to execute compared to the lower bound sieve, and the desired upper bound follows. Indeed, the count of non-degenerate integral points in $V(\bZ)$ already serves as a suitable upper bound, as we are not concerned with the exact count. Note that multiplicities can be ignored with the observation that the multiplicities of the action of $G(\bZ)$ on $V$ is uniformly bounded.  

\subsection{Proof of Theorem \ref{MT2}} Put $\N_1(X), \N_2(X), \N_3(X)$ for the cardinalities of the three families considered in Theorem \ref{MT1}, and let $\M_i(X)$ denote the corresponding count weighted by the size of the 3-Selmer group of each element. Then Theorem \ref{MT1} implies that 
\[\N_i(X) \gg X^{\frac{3}{4}} \text{ for } i = 1,2,3.\]
Now, by (\ref{3-Sel count}) the number of 3-Selmer elements in each of the families is $O\left(X^{\frac{3}{4}} \right)$. Therefore 
\[\limsup_{X \rightarrow \infty} \frac{\M_i(X)}{\N_i(X)} \ll \limsup_{X \rightarrow \infty} \frac{X^{3/4}}{X^{3/4}} < \infty,\]
as claimed.

\end{document}